\documentclass[11pt]{article}
\usepackage{amsfonts}
\usepackage{amsmath,amsthm}
\usepackage{hyperref}
\usepackage{slashed}
\usepackage[all]{xy}
\usepackage{graphicx}

\usepackage{amssymb}
\usepackage{latexsym}

\DeclareMathAlphabet{\mathpzc}{OT1}{pzc}{m}{it}
\DeclareMathAlphabet\EuFrak{U}{euf}{m}{n}	
\SetMathAlphabet\EuFrak{bold}{U}{euf}{b}{n}	

\newcommand{\Aut}{ {\bf aut \, } }

\newcommand{\bs}{\boldsymbol}
\newcommand{\bsh}{{\boldsymbol h}}
\newcommand{\bsx}{{\boldsymbol x}}
\newcommand{\bsy}{{\boldsymbol y}}

\newcommand{\bC} {{\mathbb C}}

\newcommand{\bR} {{\mathbb R}}

\newcommand{\bU} {{\mathbb U}}

\newcommand{\bZ} {{\mathbb Z}}

\newcommand{\bN} {{\mathbb N}}

\newcommand{\bP} {{\mathbb P}}

\newcommand{\eps}{\epsilon}

\newcommand{\mA}{\mathcal A}
\newcommand{\mB}{\mathcal B}
\newcommand{\mC}{\mathcal C}
\newcommand{\mD}{\mathcal D}

\newcommand{\mF}{\mathcal F}
\newcommand{\mG}{\mathcal G}
\newcommand{\mH}{\mathcal H}

\newcommand{\mO}{\mathcal O}

\newcommand{\mS}{\mathcal S}

\newcommand{\mU}{\mathcal U}

\newcommand{\mW}{\mathcal W}

\newcommand{\efh}{\EuFrak h}

\newcommand{\efB}{\EuFrak{B}}

\newcommand{\efF}{\EuFrak{F}}

\newcommand{\efK}{\EuFrak{K}}

\newcommand{\lb}{{\left\langle \right.}}
\newcommand{\rb}{{\left. \right\rangle}}

\newcommand{\Dom}{ \, {\mathrm{Dom}} \, }

\newcommand{\supp}{{\mathrm{supp}} \, }

\newtheorem{thm}{Theorem}[section]

\newtheorem{lem}[thm]{Lemma}
\newtheorem{prop}[thm]{Proposition}

\newtheorem{defn}[thm]{Definition}

\newtheorem{rem}[thm]{Remark}

\theoremstyle{definition}
\newtheorem{ex}{Example}[section]

\theoremstyle{remark}

\numberwithin{equation}{section}

\begin{document}

\author{\textsc{Ezio Vasselli} \\
\small{Dipartimento di Matematica, Universit\`a di Roma ``Tor Vergata'',}\\
\small{Via della Ricerca Scientifica, I-00133 Roma,  Italy.}  \\
\small{\texttt{ ezio.vasselli@gmail.com  }}\\[20pt]
}

\title{ Twisted tensor products of field algebras }
\maketitle

\begin{abstract} 
Let $\mA$ be a C*-algebra, $\bsh$ a Hilbert space and $\mC_\bsh$ the CAR algebra over $\bsh$.
We construct a twisted tensor product of $\mA$ by $\mC_\bsh$ such that the two
factors are not necessarily one in the relative commutant of the other.
The resulting C*-algebra may be regarded as a generalized CAR algebra constructed 
over a suitable Hilbert $\mA$-bimodule.
As an application, we exhibit a class of fixed-time models where a free Dirac field (giving rise to the $\mC_\bsh$ factor)
in general is not relatively local to a free scalar field (which yields the $\mA$ factor).
In some of the models, gauge-invariant combinations of 
the two (not relatively local) fields form a local observable net. 

\medskip

\noindent 
{\bf Mathematics Subject Classification.} 46L06, 46L08, 81T05 \\
{\bf Keywords.} C*-algebras, Hilbert modules, Fock spaces, quantum field theory
\end{abstract}



\section{Introduction}
\label{sec.intro}

In quantum field theory, a common way to construct a field system is to perform 
the tensor product of Fock spaces and define for each factor the corresponding free field. 
This method is used in both exactly solvable models and perturbative theory,
as a preliminary step to perform Wick or time-ordered products
and get the corresponding Wightman or Green functions.
At the algebraic level the typical construction is given by the spatial tensor product 
$\mF \doteq \mC_\bsh \otimes \mA$,
where $\mC_\bsh$ is the CAR algebra over a Hilbert space $\bsh$ (generalized free Fermi field)
and $\mA$ is a CCR algebra (generalized free Bosonic field).

\medskip 

The starting point of the present paper is the remark that we may regard $\mF$ 
as the C*-algebra generated by a Fermi field $\psi$ intended in a broader sense,
defined in terms of creation and annihilation operators living in a 
fermionic Fock bimodule $\efF_-(\efh)$.
Here, $\efh \doteq \bsh \otimes \mA$ is the free Hilbert bimodule
carrying the \emph{trivial left action} uniquely determined by the relations
$wA = Aw$, $w \in \bsh$, $A \in \mA$.
Adopting this point of view, $\mA$ appears as part of the target space for 
the "non-commutative anticommutator function" of $\psi$.

\medskip 

This suggests a strategy to escape from tensor products and possibly produce
interesting models. The idea is that of considering 
Hilbert bimodules carrying a non-trivial left action,
with the aim of constructing field systems where the Bosonic components do not 
necessarily commute with the fermionic ones. 
This is what happens, for example, in QED, 
where (in positive gauges) charged fields cannot be relatively local to
the electromagnetic field \cite{FPS74}. 
%

\medskip

It is aim of the present paper to provide an algebraic machinery to
construct field C*-algebras having the above mentioned property.
We start by reviewing generic Hilbert bimodules, their Fock bimodules 
and the corresponding GNS Hilbert spaces obtained by applying states of $\mA$ \S \ref{sec.0}. 
On these spaces there are well-defined annihilation and creation operators, 
and the "zero particle space" is the GNS space of $\mA$ 
instead of $\bC$ as in usual Fock space.
Annihilation operators --while they perform the usual operation 
of annihilating states defined by $\efh$-- modify (without annihilating) 
states in the GNS space of $\mA$.
Thus in the present paper the term \emph{annihilation operator} should be intended 
in this broader sense.

\medskip 

To define fermionic spaces, one should introduce a permutation symmetry on Fock bimodules.
Since in general this is an impossible task, 
%
%
%
we focus on free Hilbert bimodules
for which it is possible to define a permutation symmetry in the obvious way \S \ref{sec.B}.
This does not necessarily lead to the trivial construction $\mF = \mC_\bsh \otimes \mA$,
because the crucial property to avoid it is non-triviality of the left $\mA$-action.
Having defined our fermionic Fock space, we face the fact that antisymmetry
\begin{equation}
\label{eq.intro.Pauli}
f \otimes_- g \ = \ - g \otimes_- f \ \ \ , \ \ \ f,g \in \efh \, ,
\end{equation}
does not hold in full generality. The reason relies on commutation properties of $\mA$: 
namely, there are suitable support C*-subalgebras $\mA(f),\mA(g)$ of $\mA$, 
and a sufficient condition for (\ref{eq.intro.Pauli}) holding true is that
\begin{equation}
\label{eq.intro.Pauli'}
[ \mA(f) \, , \, \mA(g) ] \ = \ 0
\end{equation}
together with the fact that the left $\mA(g)$-action on $f$ is trivial
and analogously for $\mA(f)$ on $g$, Remark \ref{rem.Pauli}.
In this case we say that $f$ and $g$ are \emph{mutually free}.
 
An other consequence of the fact that we are dealing with Hilbert modules
rather than Hilbert spaces is given by non-stability of the fermionic Fock module
both under the action of creation operators and the left $\mA$-action.
Both these problems are solved by considering a particular class of left $\mA$-actions 
that we call \emph{$\mG$-twists}, where $\mG$ is a group of unitary generators of $\mA$
(Lemma \ref{lem.Gtwist}).
$\mG$-twists are given by group morphisms $u : \mG \to \mU(\bsh)$
that are used to perturb the trivial left $\mA$-action, Def.\ref{def.twist}.
We exhibit two classes of C*-algebras that naturally admit $\mG$-twists:
the first is given by  Weyl algebras (giving rise to generalized free Bosonic fields),
Example \ref{ex.twist2};
the second, Example \ref{ex.twist3}, is the universal C*-algebra of the electromagnetic field
\cite{BCRV}.

\medskip 

Finally in \S \ref{sec.CAR} we construct our generalized CAR algebra having as input the
free module $\efh$ and the twist $u$, Theorem \ref{thm.A0.1}. 
Adding as a third ingredient a suitable conjugation $\kappa$ acting on $\bsh$ we obtain what 
we call a \emph{Dirac triple over $\mA$}, from which we construct a Dirac field (\ref{eq.FF01})
and the corresponding field algebra (\ref{eq.FF02.sd}).
It is this field algebra that defines the desired twisted tensor product of $\mA$ by $\mC_\bsh$. 

As an application we exhibit a family of models in a fixed-time formulation, 
depending on a tempered distribution $\sigma \in \mS'(\bR^3)$ that defines the twist. In our models,
a free Dirac field (generating the C*-algebra $\mC_\bsh$) and a free scalar field
(generating a Weyl C*-algebra $\mW$) give rise to the following situations:
(1) $\mC_\bsh$ and $\mW$ do not commute, but are relatively local (when $\sigma$ is the Dirac delta);
(2) $\mC_\bsh$ and $\mW$ are not relatively local (for $\sigma$ having support with non-empty interior);
(3) $\mW$ is not relatively local to $\mC_\bsh$, but is in the commutant of the fixed-point
algebra of $\mC_\bsh$ under the gauge action (when $\sigma$ is the Lebesgue measure).

\medskip 

A discussion of our results, work in progress and perspectives are given in the final section \S \ref{sec.C}.

\medskip

In the following points we fix some conventions:
\begin{itemize}
\item Throughout the present paper, $\mA$ will denote a C*-algebra with unit ${\bf 1}$
      and $\mS(\mA)$ its state space.
      Unless otherwise stated, 
      Hilbert space representations and *-morphisms are assumed to be non-degenerate (hence unital).
\item We reserve Euler Fraktur fonts for Hilbert modules, related objects 
      ($\efh$, $\efB(\efh)$, $\efF(\efh)$ and so on),
      and objects defined starting from Hilbert modules (for example, the Hilbert spaces $\efh^\omega$
      for $\omega \in \mS(\mA)$).
      Instead, unless otherwise stated the calligrafic font is used for Hilbert spaces 
      ($\mH$) and C*-algebras ($\mA$, $\mB(\mH)$, ...).
      As an exception to this rule, we adopt a lower case bold font for 
      "one-particle Hilbert spaces", that therefore are denoted by $\bsh$.
\item Free Hilbert modules usually appear in literature with the notation 
      $\efh = \bsh \otimes \mA$, where $\bsh$ is a Hilbert space.
      To use the symbol $\otimes$ without ambiguities, in the sequel 
      up to rare exceptions we will write
      $\bsh \otimes \mA  \equiv  \bsh\mA$ and $v \otimes A \equiv vA$ for
      $v \in \bsh$ and $A \in \mA$. The symbol $\otimes$ shall be used for
      the internal tensor product of Hilbert bimodules or the 
      Hilbert space tensor product.
\item For generic elements of $\efh = \bsh\mA$ we will often use the notation 
      of implicit sum for repeated indices, $\sum_i v_i A_i \equiv v_i A_i$.
      The sum may be infinite, and in this case our notation should be understood
      in terms of norm convergence in $\efh$.
\end{itemize}

\section{Fock spaces over C*-algebras}
\label{sec.0}

In the present section we collect some standard facts on Hilbert bimodules that we present in a form 
suitable for our purposes. 
%
%
References are Blackadar's book \cite{Bla},
and the papers \cite{Kas80},\cite{AM19,Pim96,Ske} for GNS-representations and Fock bimodules respectively.
For the use of Hilbert bimodules in mathematical physics 
see \cite{ALV97,Ske,Ske98}, \cite{DR89,DR90},
\cite{BL1,BL2}, and the recent review \cite{AM19}.

\paragraph{Hilbert modules.}
A \emph{(right Hilbert) $\mA$-module} is a complex vector space $\efh$
carrying a right $\mA$-module action and endowed with a $\mA$-valued, strictly positive and 
right $\mA$-linear scalar product $\lb \cdot , \cdot \rb$, 
which induces the norm $\| \lb \cdot , \cdot \rb \|$ under which $\efh$ is a Banach space. 
It is worth noting that then
\begin{equation}
\label{eq.0.00}
A \lb f,g \rb \, = \, 
( \lb g,f \rb A^* )^* \, = \, 
( \lb g,fA^* \rb )^* \, = \, 
\lb fA^* , g \rb \, ,
\end{equation}
for all $f,g \in \efh$, $A \in \mA$.
%
%
An $\mA$-module $\efh$ defines the C*-algebra $\efB(\efh)$ of linear operators
$T : \efh \to \efh$ such that there is $T^* \in \efB(\efh)$ with 
$\lb T^* f, g \rb = \lb f , Tg \rb$, $f,g \in \efh$.
This implies that $T \in \efB(\efh)$ is right $\mA$-linear, \emph{i.e.}
$T(fA) = (Tf)A$, and bounded.
The C*-algebra of compact operators is given by the closed ideal $\efK(\efh) \subseteq \efB(\efh)$
generated by the elementary operators
$|f\rb\lb g| h \doteq f \lb g,h \rb$, $h \in \efh$.
We say that $\efh$ is a \emph{Hilbert $\mA$-bimodule} whenever there is a *-morphism
$\lambda : \mA \to \efB(\efh)$,
called the left action (Hilbert bimodules are also known as \emph{C*-correspondences} \cite{AM19}). 
In the present paper we shall assume that $\lambda$ is faithful and non-degenerate,
and following a standard notation we will write 
\[
Af \, \equiv \, \lambda(A)f
\ \ \ , \ \ \ 
\forall f \in \efh \, , \, A \in \mA 
\, .
\]
Let $\omega \in \mS(\mA)$ be a state. Then $\lb f,g \rb_\omega \doteq \omega(\lb f,g \rb)$
is a scalar product and defines the Hilbert space $\efh^\omega$ whose elements
$f^\omega \in \efh^\omega$ are defined in correspondence with $f \in \efh$.
We have the representation
\begin{equation}
\label{eq.0.01}
\pi^\omega : \efB(\efh) \to \mB(\efh^\omega)
\ \ \ , \ \ \ 
T^\omega f^\omega \, \doteq \, (Tf)^\omega \, .
\end{equation}
The argument to prove that (\ref{eq.0.01}) is well-defined is standard,
so we give just a sketch.
Let $T \in \efB(\efh)$ and $v \in \efh$ such that $v^\omega = 0$, {\it i.e.} $\lb v,v \rb_\omega = 0$; 
then $\| T \|^2 - T^*T$ is a positive element of $\efB(\efh)$ and this 
implies that $\lb v , (\| T \|^2 - T^*T)v \rb \in \mA$ is positive for any $v \in \efh$.
As a consequence 
\[
0 \ = \ 
\| T \|^2 \lb v,v \rb_\omega \ \geq \ 
\lb v , T^*T v \rb_\omega \ = \
\lb Tv , Tv \rb_\omega \ \geq \ 0 \, ,
\]
implying $(Tv)^\omega = 0$,
%
%
and $T^\omega$ is a well-defined linear operator on $\mB(\efh^\omega)$ with norm $\leq \| T \|$.
Note that (\ref{eq.0.01}) restricts to the representation of the left $\mA$-action
\begin{equation}
\label{eq.0.02}
\pi^\omega : \mA \to \mB(\efh^\omega) 
\ \ \ , \ \ \ 
A^\omega f^\omega \, \doteq \, (Af)^\omega \, .
\end{equation}
We call (\ref{eq.0.01}) and (\ref{eq.0.02}) the GNS-representations induced by $\omega$.

\begin{ex}
\label{ex.0.1}
We set $\efh \doteq \mA$ with right (left)
$\mA$-module action given by right (left) multiplication and scalar product
$\lb f,g \rb \doteq f^*g$, $f,g \in \efh$.
If $\omega \in \mS(\mA)$, then it is readily seen that (\ref{eq.0.02}) is the 
GNS representation in the usual sense. For future use, we write $v_A^\omega$
to indicate the GNS vector corresponding to $A \in \mA$, so that 
$\| v_A^\omega \|^2 = \omega(A^*A)$.
\end{ex}

\begin{ex}
\label{ex.0.2}
Let $\efh \doteq \bsh \mA$ denote the free right Hilbert $\mA$-module, 
with right action
$vB \doteq w \otimes AB$
and scalar product 
$\lb v,v' \rb \doteq A^*A' \lb w,w' \rb$,
where $v \doteq w \otimes A$, $v' \doteq w' \otimes A'$, $w,w' \in \bsh$, $A,A',B \in \mA$.
If
$\lambda : \mA \to \efB(\efh) \simeq \mB(\bsh) \otimes \mA$
is a *-morphism, then we have the left $\mA$-action $Af \doteq \lambda(A)f$.
Given $\omega \in \mS(\mA)$, the Hilbert space $\efh^\omega$ 
is isomorphic to $\bsh \otimes \bsh^\omega$, where $\bsh^\omega$
is the GNS space of $\omega$.
\end{ex}

\paragraph{The Fock bimodule.}
Let $n \in \bN$ and $\efh$ denote a Hilbert bimodule.
We consider the tensor product of complex vector spaces $\efh^{\odot n}$, 
that we endow with the scalar product
%
\begin{equation}
\label{eq.B.01}
\lb v , w \rb 
\, \doteq \, 
\lb v_n \, , \, \lb v_{n-1} \, , \, \lb \ldots \lb v_1 , w_1 \rb \ldots \rb w_{n-1} \rb \, w_n \rb 
\, ,
\end{equation}
where $v \doteq v_1 \otimes \ldots \otimes v_n$ and $w$ are in $\efh^{\odot n}$.
Note that the term $\lb v_{n-1} , \ldots w_{n-1} \rb$ appearing in the previous expression
belongs to $\mA$, so it makes sense to consider its left action on $w_n$,
and analogously for the nested terms $\lb v_{n-k} , \ldots w_{n-k} \rb$, $k=2,\ldots,n-1$.
The completion obtained by (\ref{eq.B.01}) is denoted by $\efh^n$ and is called 
the \emph{internal $n$-fold tensor product}: it is endowed with the $\mA$-module actions
\[
Av \doteq Av_1 \otimes \ldots \otimes v_n
\ \ \ , \ \ \ 
vA \doteq v_1 \otimes \ldots \otimes v_nA \, ,
\]
and therefore it forms a Hilbert $\mA$-bimodule. By construction, it turns out
\begin{equation}
\label{eq.B.01'}
\ldots \otimes v_k \otimes A v_{k+1} \otimes \ldots
\, = \, 
\ldots \otimes v_k A \otimes v_{k+1} \otimes \ldots
\ \ \ , \ \ \ 
\forall A \in \mA \, .
\end{equation}
\begin{rem}
\label{rem.nonFock1}
A delicate point of the tensor product of Hilbert bimodules is that
an elementary tensor $v \in \efh^{\odot n}$ may have norm zero even when the all the factors 
$v_1 ,  \ldots , v_n$ do not. For example, if $\mA = C(X)$ is commutative,
$\efh$ is the module of sections of a vector bundle $E \to X$, and $Af = fA$ 
is defined by pointwise multiplication for all $A \in C(X)$ and $f \in \efh$,
then $\lb v,v \rb = \lb v_1,v_1 \rb\lb v_2,v_2 \rb = 0$ 
for $v \doteq v_1 \otimes v_2$ and $\supp(v_1) \cap \supp(v_2) = \emptyset$.
%
%
\end{rem}

The direct sum of Hilbert $\mA$-bimodules $\efh , \efh'$ is defined in the obvious way,
and yields the $\mA$-bimodule $\efh \oplus \efh'$.
Thus we define the \emph{Fock $\mA$-bimodule}
\begin{equation}
\label{eq.0.03}
\efF(\efh) \, \doteq \, \bigoplus_{n=0}^\infty \efh^n \, ,
\end{equation}
with $\efh^0 \doteq \mA$ and $\efh^1 \doteq \efh$; any $v \in \efF(\efh)$ is a sequence
$v = \{ v^n \in \efh^n \}$.
Let now $f \in \efh$, $n \geq 1$, and $v \doteq v_1 \otimes \ldots \otimes v_n \in \efh^n$. We set
\begin{equation}
\label{eq.0.04}
\left\langle f \right| v \, \doteq \, \lb f , v_1 \rb v_2 \otimes \ldots \otimes v_n \, ,
\end{equation}
obtaining a right $\mA$-linear operator 
$\left\langle f \right| : \efh^n \to \efh^{n-1}$.
We define the annihilation operator
\begin{equation}
\label{eq.0.05}
(a(f)v)^{n-1}  \, \doteq \, \sqrt{n} \, \, \left\langle f \right| v^n
\ \ \ , \ \ \
n \geq 1 \, ,
\end{equation}
with domain
$\Dom(a(f)) \doteq \efF^\#(\efh) \doteq
\left\{ v \in \efF(\efh) \,: \, \exists \sum_n n \lb v^n , v^n \rb \in \mA \right\}$.
Next we define the creation operator
\begin{equation}
\label{eq.A0.1.3}
\left\{ 
\begin{array}{l} 
(a^*(f) v)^{n+1} \, \doteq \, 
\sqrt{n+1} \, \, f\otimes v^n \ \ \ , \ \  n \geq 1 \, , 
\\
(a^*(f)v)^1 \doteq fv^0 
\ \ \ \ \ \ \ \ \ \ \ \ \ \ \ \ \ \ \ \ \ \, , \ \ n=0 \, , 
\end{array} 
\right. 
\end{equation} 
for $v \in \Dom(a^*(f)) \doteq \efF^\#(\efh)$. Note that $v^0 \in \mA$, so the expression 
$fv^0 \in \efh$ makes sense. Moreover, if $v^n = 0$ for all $n \geq 1$ then 
$a(f)v = 0$, so that $\mA = \efh^0$ is in the kernel of $a(f)$. We have 
\begin{equation}
\label{eq.CrAn}
a(f) a^*(g) \, = \, (n+1) \lb f,g \rb  \ \ \ , \ \ \ 
a^*(g) a(f) \, = \, n g \otimes \left\langle f \right|  \ \ \ , \ \ \ 
f,g \in \efh \, . 
\end{equation}
The above notation suggests that $a^*(f)$ is the adjoint of $a(f)$ and,
actually, for elementary tensors $v \in \efh^n$, $w \in \efh^{n-1}$, we have 
\begin{equation}
\label{eq.CrAn'}
\begin{array}{l}
\lb v , a^*(f) w \rb  \ = \ 
\sqrt{n} \, \lb v , f \otimes w \rb \ = \\ 
\sqrt{n} \, \lb v_2 \otimes \ldots \otimes v_n \, , \, \lb v_1 , f \rb w \rb \ = \ 
\sqrt{n} \, \lb \lb f , v_1 \rb v_2 \otimes \ldots \otimes v_n \, , \, w \rb \ = \\
\lb a(f) v \, , \, w \rb \, .
\end{array}
\end{equation}
Thus $a^*(f)$ behaves as the adjoint of $a(f)$ on the common domain $\efF^\#(\efh)$.
Note that $a^*(f)$ and $a(f)$ are unbounded, thus we should be careful 
when we use the term \emph{adjoint}.
Yet, since we are interested in the fermionic case
where it will be shown that the creation and annihilation operators are bounded,
we prefer to not discuss this point here.


\paragraph{Hilbert spaces in Hilbert modules.}
A \emph{Hilbert space in $\efh$} is given by a closed vector subspace $\bsh \subset \efh$
such that 
$\lb w , w' \rb \in \bC {\bf 1}$, $\forall w,w' \in \bsh$.
The proof of the following result is trivial, therefore it is omitted.
\begin{lem}
\label{lem.0.1}
Let $\bsh \subset \efh$ be a Hilbert space. 
Then there is an injective linear mapping $\mF(\bsh) \to \efF(\efh)$
preserving the scalar product,
where $\mF(\bsh)$ is the Fock space in the usual sense.
\end{lem}

\paragraph{GNS-Hilbert spaces of the Fock bimodule.}
Let now $\omega \in \mS(\mA)$ and $\efh^{n,\omega}$ denote the Hilbert space
obtained by completion of $\efh^n$ under the scalar product $\lb \cdot , \cdot \rb_\omega$.
Moreover, let $\efF^\omega(\efh)$ denote the Hilbert space obtained by the analogous completion of $\efF(\efh)$. 
Basic properties of this Hilbert space are resumed in the following result.
\begin{prop}
\label{prop.0.1}
Let $\efh$ be a Hilbert $\mA$-bimodule and $\omega \in \mS(\mA)$. 
With the above notation, there is a decomposition
\begin{equation}
\label{eq.0.07}
\efF^\omega(\efh) \, \simeq \, \bigoplus_{n=0}^\infty \efh^{n,\omega} \, .
\end{equation}
The component $\efh^{0,\omega}$ is the usual GNS Hilbert space of $\omega$, 
and there is a representation
\begin{equation}
\label{eq.0.07'}
\hat{\pi}^\omega : \mA \to \mB(\efF^\omega(\efh))
\ \ , \ \ 
\hat{\pi}^\omega \, = \, \oplus_n \pi^{n,\omega}
\, ,
 \end{equation}
where each $\pi^{n,\omega}$ is defined as in (\ref{eq.0.02}) and 
$\pi^{0,\omega} : \mA \to \mB(\efh^{0,\omega})$ is the usual GNS representation.
Given a Hilbert space $\bsh \subset \efh$, 
the closed vector subspace of $\efF^\omega(\efh)$ spanned
by elements of the type $(wA)^\omega$, $w \in \bsh^n$, $A \in \mA$, $n \in \bN$,  
is a Hilbert space isomorphic to $\mF(\bsh) \otimes \efh^{0,\omega}$, 
where $\mF(\bsh)$ is the Fock space in the usual sense,
so that there is an embedding
\begin{equation}
\label{eq.0.07''}
\iota : \mF(\bsh) \otimes \efh^{0,\omega} \to \efF^\omega(\efh) 
\ \ , \ \ 
\iota(w \otimes v_A^\omega) \, \doteq \, (wA)^\omega
\, .
\end{equation}
%
\end{prop}

\begin{proof}
The direct sum decomposition (\ref{eq.0.07}) is obvious.
About $\efh^{0,\omega}$, by definition it is given by the completion of $\efh^0 = \mA$
under the scalar product $\omega(A^*A')$, $A,A' \in \mA$, so it is the GNS-space of $\omega$. 
The decomposition (\ref{eq.0.07'}) trivially follows by the previous points, 
as well as the fact that $\pi^{0,\omega}$ is the GNS representation of $\omega$.
%
%
Finally, the embedding (\ref{eq.0.07''}) follows by the fact that,
given elementary tensors $w,w' \in \bsh^n$, $A,A' \in \mA$,
\[
\lb wA , w'A' \rb \ = \ A^* \, \prod_k \lb w_k , w'_k \rb \, A' \, , 
\]
having used the fact that
\[
\lb w_n , \lb w_{n-1} , \lb \ldots \lb w_1 , w'_1 \rb \ldots \rb w'_{n-1} \rb w'_n \rb
\ = \ 
\prod_k \lb w_k , w'_k \rb
\]
(recall that $\lb w_k , w'_k \rb \in \bC$ for all $k$).
\end{proof}

\begin{rem}[Non-Fock nature of $\efF^\omega(\efh)$]
\label{rem.prop.01}
We emphasize that in general $\efh^{n,\omega}$ is not the $n$-fold tensor product $(\efh^\omega)^{\otimes n}$
in the sense of Hilbert spaces, and as a consequence $\efF^\omega(\efh)$ is not the Fock space of $\efh^\omega$.
This is readily seen by comparing on elementary tensors $v,w \in \efh^n$ 
the scalar product obtained by (\ref{eq.B.01}) ,
\begin{equation}
\label{eq.1.rem.prop.01}
\lb v^\omega , w^\omega \rb
\, \doteq \, 
\lb v_n \, , \, \lb v_{n-1} \, , \, \lb \ldots \lb v_1 , w_1 \rb \ldots \rb w_{n-1} \rb \, w_n \rb_\omega 
\end{equation}
with the scalar product of $(\efh^\omega)^{\otimes n}$,
\begin{equation}
\label{eq.2.rem.prop.01}
\lb v^\omega , w^\omega \rb' \ \doteq \ \prod_k \lb v_k , w_k \rb_\omega \, ,
\end{equation}
the latter corresponding to the Fock space of $\efh^\omega$. 
%
%
\end{rem}


\begin{rem}
\label{rem.prop.01''}
Let $\bsh \subset \efh$ be a Hilbert space.
Then the embedding of $\mF(\bsh)$ into $\efF(\efh)$ (Lemma \ref{lem.0.1})
factorizes through the isometric mapping 
\[
\mF(\bsh) \simeq \mF(\bsh) \otimes \Omega  \to  \efF^\omega(\efh) 
\ \ \ , \ \ \ 
w \otimes \Omega \mapsto w^\omega \, ,
\]
where $\Omega \in \efh^{0,\omega}$ is the GNS-vector ("vacuum")
obtained by (\ref{eq.0.07''}) for $A \equiv {\bf 1}$.
Thus $\mF(\bsh)$ embeds in any GNS space $\efF^\omega(\efh)$.
\end{rem}


\paragraph{Creation and annihilation operators in $\efF^\omega(\efh)$.}
Let $f \in \efh$ and $n \in \bN$. Then the mapping
$w \mapsto f \otimes w$
defines a right $\mA$-linear operator $T(f) : \efh^n \to \efh^{n+1}$, and since
\[
\lb (T(f) w \, , \, v \rb \ = \ 
\lb f \otimes w \, , \, v \rb \ = \ 
\lb w \, , \, \lb f,v_1 \rb (v_2 \otimes \ldots \otimes v_n) \rb \ = \ 
\lb w \, , \, \lb f | v \rb 
\, ,
\]
we find that $T(f)$ has adjoint $T(f)^* = \lb f |$ and, as a consequence, 
is bounded \cite[Prop.13.2.2]{Bla} (this also implies that $\lb f |$ is bounded and adjointable).
Therefore, the argument used to construct the representation (\ref{eq.0.02}) 
allows to define the linear operators
\[
T(f)^\omega : \efh^{n,\omega} \to \efh^{n+1,\omega}
\ \ \ , \ \ \ 
\lb f |^\omega : \efh^{n+1,\omega} \to \efh^{n,\omega}
\, .
\]
Rescaling these operators by the $\sqrt{n}$ factors in dependence of the order of the
tensor powers we obtain the creation and annihilation operators
\begin{equation}
\label{eq.0.11}
(a^\omega(f) v^\omega)^{n-1} \ \doteq \ 
\sqrt{n} \, \lb f |^\omega v^\omega \ = \ 
\sqrt{n} \, (\lb f | v)^\omega \, ,
\end{equation}

\begin{equation}
\label{eq.0.12}
\left\{
\begin{array}{l}
(a^{*,\omega}(f) v^\omega)^{n+1} \, \doteq \, 
\sqrt{n+1} \, \, (f \otimes v^n)^\omega \ \ \ , \ \  n \geq 1 \, ,
\\ \\ 
(a^{*,\omega}(f)v)^1 \, \doteq \, (fv^0)^\omega
\ \ \ \ \ \ \ \ \ \ \ \ \ \ \ \ \ \ \ \ \ \, \, , \ \ n=0 \, ,
\end{array}
\right.
\end{equation}
where $v^\omega = \{ v^{n,\omega} \}$ is defined for $v \in \efF^\#(\efh)$.
Clearly $a^{*,\omega}(f)$ is (contained in) the adjoint of $a^\omega(f)$ in the sense of Hilbert spaces,
%
%
\[
\begin{array}{lcl}
\lb w^{\omega,n-1} \, , \, (a^\omega(f) v^\omega)^{n-1} \rb_\omega & = & 
\sqrt{n} \, \lb w^{n-1} \, , \, \lb f,v_1 \rb (v_2 \otimes \ldots \otimes v_n) \rb_\omega \ = \\ & = &
\sqrt{n} \, \lb f \otimes w^{n-1} \, , \, v_1 \otimes \ldots \otimes v_n \rb_\omega \ = \\ & = &
\lb (a^{*,\omega}(f) w^\omega)^n \, , \, v^\omega \rb_\omega \, . 
\end{array}
\]
Note that since $a^*(f)$ may have non-trivial kernel (Remark \ref{rem.nonFock1}), 
the same is true for $a^{*,\omega}(f)$.

\section{Twisted module actions and permutation symmetry}
\label{sec.B}

In the present section we study the permutation symmetry on tensor powers of Hilbert bimodules
or, to be honest, the factors that in general prevent the realization of this property.
These factors are easily identified and are
non-commutativity of $\mA$ and the fact that in general
$Af \neq fA$, $A \in \mA$, $f \in \efh$.
%
%
Anyway even with the above limitations we shall give a version of the fermionic Fock space
for free Hilbert bimodules,
and define the corresponding restrictions of the creation and annihilation operators.
The tool used to perform these constructions is the one of \emph{twist},
namely a "perturbation" of the trivial left $\mA$-module action
defined using a group of unitary generators of $\mA$.
We remark that thanks to the Kasparov stabilization theorem \cite{Bla}
the choice of considering only free bimodules is not a severe requirement.

\medskip 

Let us now consider an elementary tensor $v = v_1 \otimes \ldots \otimes v_n \in \efh^n$. 
Given a permutation $\varrho \in \bP(n)$, we would like to define the operator
\begin{equation}
\label{eq.PS.01}
U_\varrho v \ \doteq \ v_{\varrho(1)} \otimes \cdots \otimes v_{\varrho(n)} \, .
\end{equation}
It is obvious that the above operator is well-defined on the algebraic elementary
tensors $v_1 \odot \ldots \odot v_n$, thus the question is what happens when one
applies the norm induced by $\mA$.
Here we immediately encounter a problem, illustrated by the following example.
Consider a Hilbert bimodule $\efh$ of the type Example \ref{ex.0.1} with $\efh = \mA = \mO_2$, 
the Cuntz algebra generated by mutually orthogonal isometries $\psi_1 , \psi_2$; 
we take $v_1 \doteq \psi_1$, $v_2 \doteq \psi_2^*$ and evaluate the scalar products
%
%
\[
\lb v_1 \otimes v_2 \, , \, v_1 \otimes v_2 \rb \ = \ 
\lb v_2 , \lb v_1 , v_1 \rb v_2 \rb \ = \ 
\psi_2 \psi_1^* \psi_1 \psi_2^* \ = \ 
\psi_2 \psi_2^* \, ,
\]
\[
\lb v_2 \otimes v_1 \, , \, v_2 \otimes v_1 \rb \ = \
\lb v_1 , \lb v_2 , v_2 \rb v_1 \rb \ = \ 
\psi_1^* \psi_2 \psi_2^* \psi_1 \ = \ 0 \, .
\]
This example shows that the above operators $U_\varrho$ in general are ill-defined,
because they map zero norm algebraic elementary tensors into tensors with norm $\neq 0$.
This also spoils the preservation of compositions of permutations in (\ref{eq.PS.01}).
For example, given the flip $\varrho(1) = 2$, $\varrho(2) = 1$, in the situation
of the previous example we find
\[
\| U_\varrho (v_2 \otimes v_1) \|  =  \| v_1 \otimes v_2 \| \neq 0
\ \ \ , \ \ \ 
U_\varrho U_\varrho (v_1 \otimes v_2)  =  0
\ \ \ , \ \ \ 
U_{\varrho \varrho} (v_1 \otimes v_2)  =  v_1 \otimes v_2 \, .   
\]

\paragraph{Free modules.}
In the language adopted in the present paper, a Hilbert bimodule $\efh$
is free if, and only if, there is a Hilbert space $\bsh \subset \efh$ 
such that
\begin{equation}
\label{eq.fM.01}
\efh \ = \ \bsh \mA \ \doteq \ 
{\mathrm{closed \ span}} \{ wA \, : \, w \in \bsh \, , \, A \in \mA \} \, .
\end{equation}
%


\begin{lem}
\label{lem.fM.01}
Let $\efh$ be a free Hilbert bimodule. Then there is an isomorphism of right
Hilbert modules $\efF(\efh) \simeq \mF(\bsh)\mA$ and, given a state
$\omega \in \mS(\mA)$, there is an isomorphism of Hilbert spaces
$\efF^\omega(\efh) \simeq \mF(\bsh) \otimes \efh^{0,\omega}$
where $\efh^{0,\omega}$ is the usual GNS Hilbert space.
The left $\mA$-module action $\lambda$ defines by evaluation the representation
\begin{equation}
\label{eq.fM.0a}
\lambda^\omega : \mA \to \mB( \mF(\bsh) \otimes \efh^{0,\omega} ) \, .
\end{equation}
\end{lem}
\begin{proof}
Let $n \in \bN$ and
$v = v_1 \otimes \ldots \otimes v_n \in \efh^n$; 
then by (\ref{eq.fM.01}) we may write $v_i = w_k A_{k;i}$, 
where $\{ w_k \}$ is an orthonormal basis of $\bsh$, $A_{k;i} \doteq \lb w_k , v_i \rb \in \mA$,
and the (possibly infinite) sum on the repeated index $k$ is performed.
Using the above decomposition, it is easily seen that $v \in \bsh^n\mA$ 
and we conclude that $\efh^n \subseteq \bsh^n\mA$.
Since the opposite inclusion is obvious, we have $\efh^n = \bsh^n\mA$,
and this proves that $\efh^n \simeq \bsh^n\mA$ implying $\efF(\efh) \simeq \mF(\bsh)\mA$. 
Thus $\efF^\omega(\efh) \simeq \mF(\bsh) \otimes \efh^{0,\omega}$ by Proposition \ref{prop.0.1}.
\end{proof}

We shall see in the following sections that
in general $\lambda^\omega(\mA)$ is not in the commutant of $\mB(\mF(\bsh))$
in $\mB( \mF(\bsh) \otimes \efh^{0,\omega} )$;
moreover, operators of the type $\lambda^\omega(A)$, $A \in \mA$, 
mix vectors of $\mF(\bsh)$ and $\efh^{0,\omega}$
in the sense that, given the GNS-vector $\Omega \doteq v_{\bf 1}^\omega \in \efh^{0,\omega}$, 
typically we will find
$\lambda^\omega(A) (v \otimes \Omega) = \sum_k v_k \otimes v^0_k$
with $v_k \neq v$ and $v^0_k \neq \Omega$.
In contrast, let us define the left $\mA$-module action
\begin{equation}
\label{eq.trivial}
A (wB) \ \doteq \ wAB
\ \ \ , \ \ \ 
w \in \bsh \, , \, A,B \in \mA
\end{equation}
(in standard notation, $A (w \otimes B) \doteq w \otimes AB$).
We call (\ref{eq.trivial}) the \emph{trivial left action}.
At the level of Fock space, it does not induce a mixing because
$\lambda(A) = 1 \otimes A$,
and clearly $\lambda^\omega(\mA)$ and $\mB(\mF(\bsh))$ commute for any 
$\omega \in \mS(\mA)$.

\paragraph{Twists and left actions.}
The following notion concerns a class of left actions
well-behaved with respect to the permutation symmetry,
obtained by twisting the trivial one by means of a group action.
\begin{defn}
\label{def.twist}
Let $\mG \subseteq \mU\mA$ be a group generating $\mA$ as C*-algebra
and $\efh = \bsh \mA$ a free bimodule.
Then the left $\mA$-action $\lambda$ is said to be $\mG$-\textbf{twisted} whenever 
there is a group morphism $u : \mG \to \mU(\bsh)$, 
that we call the $\mG$-\textbf{twist}, such that
\begin{equation}
\label{eq.def.twist}
\lambda(\gamma)w \equiv \gamma w \, = \, (u_\gamma w) \gamma 
\ \ , \ \ 
\forall \gamma \in \mG \, , \, w \in \bsh \, .
\end{equation}
We say that the $\mG$-twist is \textbf{trivial} whenever $u$ is the 
trivial representation.
\end{defn}

%
For reader's convenience we check the consistence of the previous definition by verifying 
that the r.h.s. of (\ref{eq.def.twist}) defines an adjointable 
(and as a consequence $\mA$-linear and bounded) operator.
We write 
$wA , w'A' \in \efh$ for $w,w' \in \bsh$, $A,A' \in \mA$, 
and compute
\[
\lb wA , \gamma w'A' \rb \ = \ 
A^* \lb w , u_\gamma w' \rb \gamma A' \ = \ 
(\gamma^* A)^* \lb u_\gamma^* w , w' \rb A' \ = \ 
\lb \gamma^* wA , w'A' \rb
\]
(note that $\lb w , u_\gamma w' \rb \in \bC$);
thus (\ref{eq.def.twist}) has adjoint $\gamma^*$ as expected.
We remark that since $\mG$ generates $\mA$ the left action $\lambda$ is
determined by $u$. Despite that, $u$ may have a kernel even when $\lambda$ does not
(as shall be evident in the following example, where $\mA$ can be simple).
Finally we note that in the previous definition we 
did not assume strong continuity of $u$, so in general it is not
a unitary representation; therefore we say that $u$ is a 
\emph{unitary morphism}. 

\medskip 

\begin{ex}[Weyl algebras]
\label{ex.twist2}
Let $\mS$ be a real vector space with a symplectic form $\eta$ and 
$\mW$ denote the associated Weyl C*-algebra generated by 
unitary symbols $W_s$, $s \in \mS$.
Let $\mG \subseteq \mU\mW$ denote the group generated by $W_s e^{1/2 i\theta}$
for $s \in \mS$ and $\theta \in \bR$. 
%
%
The Weyl relations imply that an Abelian quotient of $\mG$ 
is given by $\mS$ as an additive group. Thus any unitary morphism
$u_{ab} : \mS \to \mU(\bsh)$ lifts to a morphism $u : \mG \to \mU(\bsh)$ such that
$u(W_s) u(W_{s'}) = u(W_{s+s'})$
{\footnote{
As a matter of fact, such unitary morphisms can be easily obtained, 
for example by exponentials of linear functionals $\rho : \mS \to \bR$.
}}.
We then consider the free Hilbert module $\efh \doteq \bsh\mW$, and set
$W_s^\lambda (wB) \doteq (u(W_s) w) \, W_s B$, 
$s \in \mS$, $w \in \bsh$, $B \in \mW$,
obtaining unitary operators $W_s^\lambda \in \efB(\efh)$.
Since
\[
\begin{array}{lcl}
W_s^\lambda W_{s'}^\lambda wB & = &
( u(W_s) u(W_{s'})w ) \, W_s W_{s'}B \ = \\ & = &
e^{1/2i\eta(s,s')} \, ( u(W_{s+s'})w ) \, W_{s+s'}B \ = \\ & = &
e^{1/2i\eta(s,s')} \, W_{s+s'}^\lambda \, wB \, ,
\end{array}
\]
we have that $W_s^\lambda$ fulfil the Weyl relations,
so they define a *-morphism $\lambda : \mW \to \efB(\efh)$
that by construction is a left $\mW$-action twisted by $u$.
\end{ex}

\begin{ex}[The universal C*-algebra of the electromagnetic field]
\label{ex.twist3}
Let $k \in \bN$ and $\mD_k(\bR^4)$ denote the vector space of smooth, compactly supported
$k$--forms on $\bR^4$ (with Minkowski metric). 
Let $\mC_1(\bR^4) \subset \mD_1(\bR^4)$ be the subspace of 1--forms
$g = (g^\mu) \in \mD_1(\bR^4)$ such that $\delta g \doteq \partial_\mu g^\mu = 0$.
Then $\delta f \doteq -2 \partial_\nu f^{\mu\nu}$
defines a 1--form $\delta f \in \mC_1(\bR^4)$
for any $f = (f^{\mu\nu}) \in \mD_2(\bR^4)$. 
With this notation, we define the C*-algebra $\mA$ generated by 
the group $\mG$ of unitary symbols 
$V(g)$, $g \in \mC_1(\bR^4)$,
with relations
\begin{equation}
\label{sec2.01a}
V(a_1g) V(a_2g) = V((a_1 + a_2) g) \ \ , \ \ 
V(g)^* = V(-g) \ \ , \ \ 
V(0) = 1 \, ,
\end{equation}
\begin{equation}
\label{sec2.01b}
V(\delta f_1) V(\delta f_2) = V(\delta f_1 + \delta f_2 )
\ \ , \ \ 
\supp f_1 \perp \supp f_2 \, ,
\end{equation}
\begin{equation}
\label{sec2.01c}
\left[ V(g_1) , V(g_2) \right]_\bullet \, \in \mA \cap \mA'
\ \ , \ \ 
\supp g_1  \perp  \supp g_2 
\, ,
\end{equation}
\noindent where $a_1 , a_2 \in \bR$. In the above expressions we used  
the symbol $\perp$ to indicate spacelike separation, whilst
$\left[ U , V \right]_\bullet \doteq UV U^* V^*$
is the group commutator and $\mA \cap \mA'$ is the centre of $\mA$.
Any Wightman field $F(f)$, $f \in \mD_2(\bR^4)$, 
fulfilling the Maxwell equations, defines a representation 
$\pi$ of $\mA$ such that
$\pi(V(\delta f)) = e^{iF(f)}$ \cite{BCRV}.
Note that since $\mC_1(\bR^4)$, as an additive group,
is a quotient of $\mG$ under the map $V(g) \mapsto g$,
any unitary morphism of $\mC_1(\bR^4)$
induces a unitary morphism of $\mG$.
Let now $u : \mC_1(\bR^4) \to \mU(\bsh)$ be a unitary morphism;
we set $\efh \doteq \bsh\mA$ and define
$V^\lambda(g) \, wA \doteq (u_g w) V(g)A$, $w \in \bsh$, $A \in \mA$.
A straightforward check shows that the operators $V^\lambda(g)$ are in $\efB(\efh)$
and fulfil (\ref{sec2.01a}-\ref{sec2.01c}): for example, about (\ref{sec2.01b})
we have
\[
\begin{array}{lcl}
V^\lambda(\delta f_1) V^\lambda(\delta f_2) \, wA & = & 
(u_{\delta f_1} u_{\delta f_2} w) \, V(\delta f_1) V(\delta f_2) A \ = \\ & = & 
(u_{\delta f_1 + \delta f_2} w ) \, V( \delta f_1 + \delta f_2 ) A \ = \\ & = & 
V^\lambda( \delta f_1 + \delta f_2 ) \, wA \, .
\end{array}
\]
Thus we have the left action 
$\lambda : \mA \to \efB(\efh)$, $\lambda(V(g)) \equiv V^\lambda(g)$,
that by construction is $\mG$-twisted.
\end{ex}

\paragraph{Permutation symmetry.}
Lemma \ref{lem.fM.01} allows to define a permutation symmetry in the obvious way, 
by extending the one defined on $\mF(\bsh)$:
\begin{equation}
\label{eq.fM.06}
U_\varrho (wA) \, \doteq \, w_\varrho A 
\ \ \ , \ \ \ 
\varrho \in \bP(n) \, , \, w \in \bsh^n \, , \, n \in \bN \, ,
\end{equation}
where $w_\varrho \in\bsh^n$ is the vector transformed under the usual permutation symmetry in Fock space.
Of course, to apply the previous definition we must express a tensor of the type
$v_1 \otimes \ldots \otimes v_n$, $v_1 , \ldots , v_n \in \efh$,
in terms of tensors of the type 
$(w_1 \otimes \ldots \otimes w_n) A$, $w_1 , \ldots , w_n \in \bsh$, $A \in \mA$.
By definition $U_\varrho$ is unitary on $\efF(\efh)$ in the sense of right Hilbert modules,
thus we get the spectral projection $P_- \doteq \oplus_n P_-^n \in \efB(\efF(\efh))$,
$P_-^n \doteq (n!)^{-1} \sum_\varrho \eps_\varrho U_\varrho$,
where $\eps_\varrho$ is the sign of $\varrho \in \bP(n)$.
We write $\efF_-(\efh) \doteq P_-\efF(\efh)$; it is obvious that $\efF_-(\efh)$ is a 
right $\mA$-module, and that it is free with
$\efF_-(\efh) \simeq \mF_-(\bsh) \otimes \mA$
where $\mF_-(\bsh)$ is the ordinary fermionic Fock space.
We have the decomposition
\begin{equation}
\label{eq.A.09}
\efF_-(\efh) \ \doteq \
\bigoplus_{n \geq 0} \efh^n_- \ \simeq \
\bigoplus_{n \geq 0} \, ( \bsh^n_- \mA )
\, ,
\end{equation}
where $\efh^n_- \doteq P_-^n \efh^n$, $\bsh^n_- \doteq P_-^n \bsh^n$. 
Of course, $\efh^0_- = \mA$ and $\efh^1_- = \efh \simeq \bsh\mA$.

\medskip

\begin{rem}[The Pauli principle]
\label{rem.Pauli}
In spite of the simplicity of our definition
some care is needed to handle antisymmetric tensors. 
For example, if $f,g \in \efh$ then it is not ensured that 
$P_-(f \otimes g) = - P_-(g \otimes f)$, 
unless $f$ and $g$ belong to $\bsh$.
In particular, the Pauli principle 
\begin{equation}
\label{eq.Pauli}
P_-(f \otimes f) \ = \ 0 \ \ \ , \ \ \ f \in \efh \, ,
\end{equation}
does not hold in general, and \textbf{a priori} its validity is ensured 
only for $f \in \bsh$.
A more general class of examples for which the Pauli principle holds is the following.
Given $f \in \efh$, we define the \textbf{support}
\begin{equation}
\label{eq.Pauli2}
\mA(f) \, \doteq \, C^* \{ \lb v , f \rb \in \mA \, : \, v \in \bsh \} \, \subseteq \mA \, .
\end{equation}
The support has the property that if $f = f_h A_h$, where $\{ f_h \}  \subset \bsh$
is an orthonormal set and $A_h \in \mA$, then $A_h = \lb f_h , f \rb \in \mA(f)$ 
for all $h \in \bN$.
%
%
%
Given $f,g \in \efh$, we write 
\begin{equation}
\label{eq.Pauli.perp}
f \bowtie g \ \ \stackrel{.}{\Leftrightarrow} \ \  
fB' = B'f \ \ , \ \ gB = Bg  \ \  , \ \ \left[ B , B' \right] \ = \ 0 \, ,
\end{equation}
for all $B \in \mA(f) \, , \, B' \in \mA(g)$.
In this case we say that $f$ and $g$ are \textbf{mutually free}, and we find
\[
\begin{array}{lcl}
P_-(f \otimes g) & = &
P_-( f_h \otimes A_hg) \ = \ 
P_-( f_h \otimes g) A_h \ = \\ & = &
P_-( f_h \otimes g_k) \,  A'_k A_h  \ = \ 
- P_-( g_k \otimes f_h) \, A'_k A_h  \ = \\ & = &
- P_-( g_k \otimes f_h) \, A_h A'_k   \ = \ 
- P_-( g_k \otimes f) \, A'_k   \ = \\ & = &
- P_-(g \otimes f) \, ,
\end{array}
\]
so that
\begin{equation}
\label{eq.Pauli'}
P_-(f \otimes g) \ = \ - P_-(g \otimes f) 
\ \ , \ \ 
f \bowtie g
\, .
\end{equation}
We conclude that validity of the Pauli principle is related to commutation properties
of $\mA$ (\textbf{locality}, when $\mA$ is the C*-algebra of a Haag-Kastler net),
and properties of the left $\mA$-action, that is, the eventuality that
it is trivial on the involved vectors of $\efh$ and elements of $\mA$. 
Note for $f \in \bsh$ we have $\mA(f) = \bC$ so that
$f,g \in \bsh$ implies $f \bowtie g$.
\end{rem}

\paragraph{Fermionic creation and annihilation operators.}
To define and handle fermionic creation and annihilation operators
we make some remarks on elementary tensors, using the attention needed in the 
case of Hilbert bimodules.

We start by noting that we may arrange order $n$ permutations by considering
for any $k=1,\ldots,n$ the set of those permutations that bring 
the $k$-th object at first position.
Thus for any elementary tensor $v = wA \in \efh^n$, $w \in \bsh^n$, $A \in \mA$, 
we may write
\begin{equation}
\label{eq.A0.1.PS1}
v_- \ \doteq \ P_-^n v \ = \ 
\frac{1}{n} \sum_{k=1}^n (-1)^{k-1} \, w_k \otimes w^{(k)}_- A
\, \in \efh^n_- \ ,
%
\end{equation}
\begin{align*}
w^{(k)}_- \, \doteq \, 
\ldots \otimes_- \hat{w}_k  \otimes_- \ldots =
\frac{1}{(n-1)!} \sum_{\varrho \in \bP_{n-1,k}} 
\eps_\varrho \, w_{\varrho(1)} \otimes \ldots \otimes w_{\varrho(n)} 
\, \in \bsh^{n-1}_- \, ,
\end{align*}
where the notation $\hat{w}_k$ indicates that $w_k$ does not appear in the tensor
%
%
and $\bP_{n-1,k}$ is understood as the permutation group of the set 
$\{ 1 , \ldots , n \} \setminus \{ k \}$; the term $(-1)^{k-1}$ in (\ref{eq.A0.1.PS1})
is the sign of the transposition bringing $w_k$ at first position.
%
%
The above expression makes manifest that in general $\efF_-(\efh)$ is not stable under
the left $\mA$-action, because the terms 
$Bw_k = \sum_i w_i B_{ik}$, $B_{ik} \doteq \lb w_i , B w_k \rb \in \mA$,
typically belong to $\efh$ and induce a mixing in the tensor product.
The point is that in general $Bw_k$ does not belong to $\bsh$, 
thus we must perform the operations of the proof of Lemma \ref{lem.fM.01} to get a tensor of
the form $Bv_- = \sum_i w'_i A'_i$ with $w'_i \in \bsh^n$.
It is after this operation that we can apply the projection $P_-^n$ and get 
$P_-^nBv_- = \sum_i w'_{i,-} A_i$.
%

\medskip

In the sequel, we shall write $\efh_\circ \subseteq \efh$ for the
vector space spanned by (finite linear combinations of) vectors of the type $w\gamma$, $w \in \bsh$, 
$\gamma \in \mG$. By construction $\efh_\circ$ is dense in $\efh$.

\medskip 

We now establish some properties of antisymmetric tensors
in case of a $\mG$-twist. The most important is that, in spite of the previous
remark, the fermionic space is stable both under the $\mG$-action and antisymmetric
tensor products by vectors in $\efh_\circ$.
\begin{lem}[Twists and permutation symmetry]
\label{lem.Gtwist}
Let $\efh = \bsh\mA$ be a free Hilbert bimodule with $\mG$-twist 
$u : \mG \to \mU(\bsh)$. Then the following properties hold:
\begin{enumerate}
    \item For any $\gamma \in \mG \subseteq \mU\mA$ it turns out 
          $\gamma \efh^n_- \in \efh^n_-$, 
          so that $\efF_-(\efh)$ is stable under the left module action by $\mG$.
          %
          %
    \item Let $g = g_i \gamma_i \in \efh_\circ$ with $g_i \in \bsh$, 
          $\gamma_i \in \mG$. Then for any tensor $w_-A \in \efh^n_-$ with 
          $w_- \doteq w_1 \otimes_- \ldots \otimes_- w_n \in \bsh_-^n$, 
          $A \in \mA$, we have 
          \begin{equation}
          \label{eq.lem.Gtwist.1}
          P_-^{n+1} (g \otimes w_-A) \ = \ 
          g_i \otimes_- (u_{\gamma_i}w_1) \otimes_- \ldots \otimes_- (u_{\gamma_i}w_n)
          \cdot \gamma_i A 
          \ \in \efh^n_- \, ,
          \end{equation}
          so $\efF_-(\efh)$ is stable under antisymmetric tensor product 
          by elements of $\efh_\circ$;
    \item For $g = g_i \gamma_i \in \efh_\circ$ it turns out 
          \begin{equation}
          \label{eq.lem.Gtwist.2}
          P_-^{n+1}(g \otimes w_-A) \ = \ 
          \frac{1}{n} \, 
          \left( 
          g \otimes w_-A - \sum_{k=1}^n (-1) ^{k-1} w_k \otimes P_-^n( g \otimes w^{(k)}_- A)
          \right) \, .
          \end{equation}
\end{enumerate}
\end{lem}

\begin{proof}
\textbf{Point 1}. Let $w_-A \in \efh^n_-$ with $w = w_1 \otimes_- \ldots \otimes_- w_n$ an elementary antisymmetric tensor. 
Then for any permutation $\varrho$ it turns out
\[
\begin{array}{lcl}
\gamma w_{\varrho(1)} \otimes \ldots \otimes w_{\varrho(n)}
& = &
%
%
(u_\gamma w_{\varrho(1)}) \otimes  \ldots \otimes (u_\gamma w_{\varrho(n)}) \gamma \, .
\end{array}
\]
Thus, defining $u_\gamma^- \doteq P_-^n \cdot \otimes^n u_\gamma \in \mU(\bsh_-^n)$
we find that $\gamma w_- A = (u_\gamma^- w_-) \gamma A$ belongs to $\efh^n_-$
as claimed.
\textbf{Point 2}. With the notation of the previous point, with $g = g_i\gamma_i$ we have
\[
\begin{array}{lcl}
g \otimes w_-A
& = &
\sum_\varrho \eps_\varrho \, g_i \otimes \gamma_i
             w_{\varrho(1)} \otimes \ldots \otimes w_{\varrho(n)} A
\, = \\ & = &
\sum_\varrho \eps_\varrho \, g_i \otimes
             (u_{\gamma_i}w_{\varrho(1)}) \otimes \ldots \otimes (u_{\gamma_i}w_{\varrho(n)})
             \gamma_i A
\, = \\ & = &
g_i \otimes (u_{\gamma_i}w_1) \otimes_- \ldots \otimes_- (u_{\gamma_i}w_n)
          \cdot \gamma_i A 
\, .
\end{array}
\]
Thus, using the fact that $P_-^n$ is a right $\mA$-linear operator, we conclude that
\[
\begin{array}{lcl}
P_-^n(g \otimes w_-A) 
& = & 
P_-^n(g_i \otimes ( (u_{\gamma_i}w_1) \otimes_- \ldots \otimes_- (u_{\gamma_i}w_n) ) 
          \, \gamma_iA 
\, = \\ & = &
g_i \otimes_- (u_{\gamma_i}w_1) \otimes_- \ldots \otimes_- (u_{\gamma_i}w_n)
          \, \gamma_i A \, .
\end{array}
\]
\textbf{Point 3}. We have, by applying (\ref{eq.A0.1.PS1}) and (\ref{eq.lem.Gtwist.1}),
\[
\begin{array}{lcl}
P_-^{n+1}(g \otimes w_-A) 
& = & 
(g_i \otimes_- u_{\gamma_i}^- w_-) \gamma_i A
\, = \\ & = &
1/n \left( 
g_i \otimes u_{\gamma_i}^- w_- - 
\sum_{k=1}^n (-1)^{k-1} \, w_k \otimes ( g_i \otimes_- u_{\gamma_i}^- w^{(k)}_-)
\right) \gamma_i A
\, = \\ & = &
1/n \left( 
g_i\gamma_i \otimes w_- - 
\sum_{k=1}^n (-1)^{k-1} \, w_k \otimes ( g_i\gamma_i \otimes_- w^{(k)}_-)
\right) A
\, = \\ & = &
1/n \left( 
g \otimes w_- - 
\sum_{k=1}^n (-1)^{k-1} \, w_k \otimes ( g \otimes_- w^{(k)}_-)
\right) A \, .
\end{array}
\]
\end{proof}

Now, by (\ref{eq.A0.1.PS1}), given $f \in \efh$ we have
\begin{equation}
\label{eq.bFBs.02}
\lb f | v_- \ = \ 
\frac{1}{n}
\sum_k (\pm 1)^{k-1} 
\lb f,w_k \rb \, w^{(k)}_- A
\, \in \efh^{n-1}
\, ,
\end{equation}
thus in general we cannot say that $\lb f | v_- \in \efh^{n-1}_-$. Yet we have the following property:
\begin{lem}
\label{lem.Gtwist2}
Assume that there is a $\mG$-twist on $\efh$ and let $f \in \efh_\circ$.
Then $\lb f | v_- \in \efh^{n-1}_-$ for all $v_- = w_- A \in \efh^n_-$,
and $\efF_-(\efh)$ is stable under the action of the operator $\lb f |$.
\end{lem}

\begin{proof}
Starting from (\ref{eq.bFBs.02}) and writing $f = f_i \gamma_i$, $f_i \in \bsh$,
$\gamma_i \in \mG$, we find
\[
\begin{array}{lcl}
\lb f | v_-
& = & 
1/n
\sum_k (\pm 1)^{k-1} 
\lb f_i ,w_k \rb \, \gamma_i^* w^{(k)}_- A
\, = \\ & = &
1/n
\sum_k (\pm 1)^{k-1} 
\lb f_i , w_k \rb \, (u_{\gamma_i}^* w_1) \otimes_- \ldots \otimes_- (u_{\gamma_i}^* w_n) \, \gamma_i^* A
\, .
\end{array}
\]
Since $\lb f_i , w_k \rb \in \bC$, the last term belongs to $\efh^{n-1}_-$ as claimed. 
\end{proof}

The next computations will allow to evaluate the anti-commutation relations.
By (\ref{eq.lem.Gtwist.2}), for $f \in \efh$ and $g = g_i \gamma_i \in \efh_\circ$ we get
\begin{equation}
\label{eq.A0.1.PS1a''}
\lb f| P_-^{n+1}(g \otimes w_-A)
\ = \ 
\frac{1}{n} \, 
\left( 
\lb f,g \rb w_-A - \sum_{k=1}^n (-1) ^{k-1} \lb f,w_k \rb P_-^n( g \otimes w^{(k)}_-)A
\right)
\end{equation}
and
\begin{equation}
\label{eq.A0.1.PS1a'''}
P_-^{n+1}(g \otimes \lb f | w_-A)
\ = \ 
\frac{1}{n} \sum_k (-1)^{k-1}  P_-^n( g \otimes \lb f,w_k \rb \, w^{(k)}_- )A
\, .
\end{equation}
About the last equality, we note that due to the antisymmetrization operator $P_-^n$ 
the factor $\lb f,w_k \rb \in \mA$ appears in $i$-th position of the involved elementary tensor 
for any $i=1,\ldots,n$. 
This implies that 
(\ref{eq.A0.1.PS1a'''}) may differ from the sum at r.h.s. of (\ref{eq.A0.1.PS1a''}),
because in general $\lb f,w_k \rb$ cannot freely shift on the left of the 
involved elementary tensor.
%
%

\medskip

Let now $P_- \doteq \oplus_n P_-^n$, where $P_-^0$ and $P_-^1$ are the identity.
We introduce the notation $\efF_-^\#(\efh) \doteq \efF_-(\efh) \cap \efF^\#(\efh)$,
and for any $f \in \efh_\circ$ define the fermionic annihilation and creation operators

\medskip

\begin{equation}
\label{eq.fFs.01}
\left\{
\begin{array}{l}
{\bs a}_-(f) : \efF_-^\#(\efh) \to \efF_-^\#(\efh) 
\ \ \ , \ \ \ 
{\bs a}_-(f) \, \doteq \, a(f) \restriction \efF_-^\#(\efh)
\, ,
\\ \\
{\bs a}_-^*(f) : \efF_-^\#(\efh) \to \efF_-^\#(\efh)
\ \ , \ \ 
{\bs a}_-^*(f) \, \doteq \, P_- a^*(f) \restriction \efF_-^\#(\efh) \, .
\end{array}
\right.
\end{equation}

\medskip

%
%
%
\noindent Note that the property
\begin{equation}
\label{eq.fFs.01'}
{\bs a}_-(f)  \efF_-^\#(\efh) \, \subseteq \, \efF_-^\#(\efh)
\ \ , \ \ 
{\bs a}_-(f) \, = \, P_- {\bs a}_-(f)
\ \ , \ \ 
\forall f \in \efh_\circ
\, ,
\end{equation}
tacitly understood in (\ref{eq.fFs.01})
is a consequence of the hypothesis that there is a $\mG$-twist
and Lemma \ref{lem.Gtwist2}.
Also note that for the moment we do not know whether the fermionic creation and annihilation operators
are bounded, thus in (\ref{eq.fFs.01}) we make use of the domain $\efF_-^\#(\efh)$.
In the following result we give an interpretation of the twist $u : \mG \to \mU(\bsh)$
as an obstacle to make the creation and annihilation operators commute with elements of $\mA$.
\begin{lem}
\label{lem.Gtwist4}
Let $v \in \bsh$ and $\gamma \in \mG$.
Then 
$\gamma {\bs a}_-^*(v) = {\bs a}_-^*(u_\gamma v) \gamma$
and
$\gamma {\bs a}_-(u_\gamma^*v) = {\bs a}_-(v) \gamma$.
\end{lem}

\begin{proof}
Let $w_-A \in \efh_-^n$, $w_- \in \bsh_-^n$, $A \in \mA$. 
By the argument of the proof of Lemma \ref{lem.Gtwist}, Point 1, we find
\[
\begin{array}{lcl}
\gamma {\bs a}_-^*(v) w_-A & = & 
\gamma P_-(v \otimes w_-) A \ = \ 
( u_\gamma v  \otimes_-  u_\gamma^- w_- ) \gamma A \ = \\ & = &  
{\bs a}_-^*(u_\gamma v) (u_\gamma^- w_-) \gamma A \ = \ 
{\bs a}_-^*(u_\gamma v) \gamma w_- A \, .
\end{array}
\]
With an analogous argument the claim about ${\bs a}_-(v)$ is proved.
\end{proof}

\begin{lem}
\label{lem.Gtwist3}
Let $f \in \efh_\circ$.
Then ${\bs a}_-^*(f)$ is the adjoint of ${\bs a}_-(f)$ over the domain $\efF_-^\#(\efh)$.
\end{lem}

\begin{proof}
By applying (\ref{eq.fFs.01}) and (\ref{eq.fFs.01'}) we find
\[
\lb v , {\bs a}_-(f) v' \rb \ = \ 
\lb v , a(f) P_- v' \rb \ = \ 
\lb P_- a^*(f) v , v' \rb \ = \ 
\lb {\bs a}_-^*(f) v , v' \rb \, ,
\]
for all $v,v' \in \efF_-^\#(\efh)$. This proves the Lemma.
\end{proof}

\medskip

Let $\omega \in \mS(\mA)$. We define $\efF_-^\omega(\efh)$ as the Hilbert space
obtained by evaluation of $\efF_-(\efh)$ over $\omega$.
In the previous lines we proved that 
${\bs a}_-(f)$ and ${\bs a}_-^*(f)$ are well-defined, $\mA$-linear
and one the adjoint of the other in the sense of Hilbert modules (on $\efF_-^\#(\efh)$);
thus, by the argument used to construct $a^\omega(f)$ and $a^{*,\omega}(f)$, we define
\[
{\bs a}_-^\omega(f) \ \ \ , \ \ \ {\bs a}_-^{*,\omega}(f) \ \ \ , \ \ \ f \in \efh_\circ \, ,
\]
as the evaluations of ${\bs a}_-(f)$ and ${\bs a}_-^*(f)$ over $\omega$.
The domain of these operators is clearly given by $\efF_-^{\#,\omega}(\efh)$,
defined by evaluation of vectors in $\efF_-^\#(\efh)$.

\section{CARs and fermionic fields}
\label{sec.CAR}

In the present section we study the anti-commutation relations that our fermionic
creation and annihilation operators fulfill, and define the corresponding Dirac field
assuming the presence of a suitable conjugation.
We bring on the light some features 
that in general prevent anticommutators to be expressed only in 
terms of the given $\mA$-valued scalar product and, primarily, make anticommutators
non-local in the sense that they do not vanish even when the involved "spinors" are orthogonal.
In fact, not surprisingly, to get anticommutators of the usual form
besides orthogonality we must require 
mutual freeness Rem.\ref{rem.Pauli}.

\ 

We proceed by maintaining the assumptions of the previous section, 
so that $\efh = \bsh\mA$ is free and there is a $\mG$-twist 
$u : \mG \to \mU(\bsh)$, $\mG \subseteq \mU\mA$.
Before computing our anti-commutation relations, 
for convenience we give a notion of mutual freeness explicitly designed to handle 
vectors in $\efh_\circ$.
Let $f = f_i \gamma_i$ and $g = g_h \gamma_h \in \efh_\circ$,
with $f_i , g_h \in \bsh$, $\gamma_i , \gamma'_h \in \mG$.
We write $f \bowtie_\circ g$ whenever for all $i,h$ it turns out
\begin{equation} 
\label{eq.split}
[ \gamma_i \, , \, \gamma'_h ] \, = \, 0
\ \ , \ \ 
u_{\gamma_i} g_h \, = \, g_h
\ \ , \ \ 
u_{\gamma'_h} f_i \, = \, f_i
\, .
\end{equation}
\begin{lem}
\label{lem.conj''}
If $f , g \in \efh_\circ$ and $f \bowtie_\circ g$, then $f \bowtie g$.
\end{lem}

\begin{proof}
As a preliminary step we note that any generator $\lb w , f \rb$ of $\mA(f)$, 
defined for $w \in \bsh$, is a linear combination in $\{ \gamma_i \}$,
so $\mA(f)$ is contained in the C*-algebra generated by $\{ \gamma_i \}$.
Thus the hypothesis $f \bowtie_\circ g$ and (\ref{eq.split}) imply
$[ A,B ] = 0$ for all $A \in \mA(f)$, $B \in \mA(g)$;
moreover,
$\gamma_i g  =  (u_{\gamma_i} g_h) \gamma_i \gamma'_h =  (u_{\gamma_i} g_h) \gamma'_h \gamma_i  =  g \gamma_i$,
and analogously for $f$ and $\gamma'_h$. This implies $f \bowtie g$,
as claimed.
\end{proof}

\paragraph{Anticommutators of creation operators.}
Let $f \bowtie_\circ g \in \efh_\circ$ and 
$v_- = w_- A \in \efh_-^n$, $w_- \in \bsh_-^n$, $A \in \mA$.
Writing $f = f_i \gamma_i$, $g = g_h \gamma'_h$  we find
\[
\begin{array}{lcl}
{\bs a}_-^*(f) {\bs a}_-^*(g) w_-A & = & 
P_-^{n+2} ( f \otimes P_-^{n+1}( g \otimes w_- ) ) A \ = \\ & = & 
P_-^{n+2} ( f_i \otimes \gamma_i P_-^{n+1} ( g_h \otimes u'_h w_- ) ) \gamma'_h A \ = \\ & = & 
P_-^{n+2} ( f_i \otimes P_-^{n+1} ( g_h \otimes u_i u'_h w_- ) ) \gamma_i \gamma'_h A \ = \\ & = & - P_-^{n+2} ( g_h \otimes P_-^{n+1} ( f_i \otimes u_i u'_h w_- ) ) \gamma_i \gamma'_h A \ = \\ & = & 
- P_-^{n+2} ( g_h \otimes P_-^{n+1} ( f_i \otimes \gamma_i \gamma'_h w_- ) ) A \ = \\ & = & 
- P_-^{n+2} ( g_h \otimes P_-^{n+1} ( f \otimes \gamma'_h w_- ) ) A \ = \\ & = & 
- P_-^{n+2} ( g \otimes P_-^{n+1}( f \otimes w_- ) ) A \ = \\ & = & 
- {\bs a}_-^*(g) {\bs a}_-^*(f) w_-A \, .
\end{array}
\]
In the previous computation, for any $n \in \bN$ we defined the antisymmetric tensor powers
$u_i \doteq u(\gamma_i)^-$, $u'_h \doteq u(\gamma'_h)^- \in \mU(\bsh_-^n)$;
moreover, we used (\ref{eq.lem.Gtwist.1}) and the fact that expressions of the type 
$P_-^{n+2} ( f_i \otimes P_-^{n+1} (g_h \otimes u_i u'_h w_- ) )$
are ordinary antisymmetric tensor powers on Hilbert space, so the Pauli principle applies
and the minus sign appears when we exchange $f_i$ and $g_h$.
We conclude that

\begin{equation}
\label{eq.20a}
[ {\bs a}_-^*(f) \, , \, {\bs a}_-^*(g) ]_+ \, = \, 0
\ \ , \ \ 
f \bowtie_\circ g \in \efh_\circ \, .
\end{equation}

\noindent 
Note that the need to assume $f \bowtie_\circ g$ to make the above anticommutators vanish
is an aspect of non-validity of the Pauli principle in full generality, Remark \ref{rem.Pauli}.
%
%

\paragraph{Anticommutators of annihilation operators.}
It should not be a surprise that passing to the adjoint of (\ref{eq.20a})
the anticommutators of annihilation operators vanish. Anyway it is instructive to perform
the explicit computations, both to understand the interplay of 
elements of $\mG$ with the relation $\bowtie_\circ$
and to keep in evidence details that may be useful to approach (unbounded) 
bosonic annihilation operators.

\medskip 

We maintain the hypothesis 
$f \bowtie_\circ g \in \efh_\circ$ and the notation for the orthonormal decompositions
$f=f_i \gamma_i$, $g=g_h\gamma'_h$.
As a first step, we further apply (\ref{eq.A0.1.PS1}) and get
\begin{equation}
\label{eq.vhk}
w_- \, = \, 
\frac{1}{n(n-1)} 
\left(
\sum_{k < h} (-1)^{k+h-1} w_k \otimes w_h \otimes w^{(h,k)}
+
\sum_{k > h} (-1)^{k+h} w_h \otimes w_k \otimes w^{(h,k)} 
\right)
\end{equation}
\begin{equation}
\label{eq.vhk'}
= \, 
\frac{2}{n(n-1)} 
\sum_{k < h} (-1)^{k+h-1} (w_k \otimes_- w_h) \otimes w^{(h,k)} \, ,
\ \ \ \ \ \ \ \ \ \ \ \ \ \ \ \ \ \ \ \ \ \ \ \ \ \ \ \ \ \ \ \ \ \, \, \, \, 
\end{equation}
where
\begin{equation}
\label{eq.vhkdef}
w^{(h,k)} \, \doteq \, 
\left\{
\begin{array}{l}
\ldots \hat{w}_k \otimes_- \ldots \otimes_- \hat{w}_h  \ldots \ \ , \ \ k<h \, , \\ 
\ldots \hat{w}_h \otimes_- \ldots \otimes_- \hat{w}_k  \ldots \ \ , \ \ k>h \, ;
\end{array}
\right.
\end{equation}
in (\ref{eq.vhk}) we used the fact that for $k < h$ the term $w_h$
needs a cyclic permutation with order $h-2$ to shift on the left of the tensor 
$w_1 \otimes \ldots \hat{w}_k \ldots \otimes w_n$, whilst for $k > h$ the order is $h-1$.
Using elementary properties of symmetric tensors, we find
%
%
$w^{(h,k)} = - w^{(k,h)}$, $h \neq k$.
%
%
After these preparations, we get
\begin{align*}
\lb g| \lb f| v_- & = \, 
\frac{2}{n(n-1)} \sum_{k < h} (-1)^{k+h-1}
\lb f \otimes g \, , \,  w_k \otimes_- w_h \rb
\, w^{(h,k)} A \\ & = \,
\frac{2}{n(n-1)} \sum_{k < h} (-1)^{k+h-1}
\lb f \otimes g \, , \,  w_k \otimes_- w_h \rb
\, w^{(h,k)} A \, .
\end{align*}
We analyze in details the terms
\begin{equation}
\label{eq.CAR.4a}
\begin{array}{l}
\lb f \otimes g \, , \, w_h \otimes_- w_k \rb \ = \\
%
%
\lb g_h \gamma'_h \, , \, \lb f_i \gamma_i , w_h \rb \, w_k \rb - 
\lb g_h \gamma'_h \, , \, \lb f_i \gamma_i , w_k \rb \, w_h \rb \, = \\
%
%
{\gamma'_h}^* \lb g_h \gamma_i  \, , \, \lb f_i , w_h \rb \, w_k \rb - 
{\gamma'_h}^* \lb g_h \gamma_i  \, , \, \lb f_i , w_k \rb \, w_h \rb \, = \\
{\gamma'_h}^* \gamma_i^*
( \lb f_i , w_h \rb  \lb g_h , w_k \rb - 
  \lb f_i , w_k \rb  \lb g_h , w_h \rb ) \, = \\
{\gamma'_h}^* \gamma_i^*
( \lb f_i \, , \, \lb g_h , w_k \rb w_h \rb  - 
  \lb f_i \, , \, \lb g_h , w_h \rb w_k \rb   ) \, = \\
\lb f_i \gamma_i \gamma'_h \, , \, \lb g_h , w_k \rb w_h \rb  - 
\lb f_i \gamma_i \gamma'_h \, , \, \lb g_h , w_h \rb w_k \rb  \, = \\
\lb \gamma'_h f  \, , \, \lb g_h , w_k \rb w_h \rb  - 
\lb \gamma'_h f  \, , \, \lb g_h , w_h \rb w_k \rb  \, = \\
\lb  f  \, , \, \lb g , w_k \rb w_h \rb  - 
\lb  f  \, , \, \lb g , w_h \rb w_k \rb  \, = \\
\lb  g \otimes f  \, , \, w_k \otimes w_h \rb  - 
\lb  g \otimes f  \, , \, w_h \otimes w_k \rb \, = \\
- \lb g \otimes f \, , \, w_h \otimes_- w_k \rb
\, ,
\end{array}
\end{equation}
and conclude that 
$( \lb g| \lb f| + \lb f| \lb g| ) v_- = 0$,
obtaining
%
%
%
%
%
%
%

\begin{equation}
\label{eq.20b}
[ {\bs a}_-(f) \, , \, {\bs a}_-(g) ]_+ \, = \, 0
\ \ , \ \ 
f \bowtie_\circ g \in \efh_\circ \, .
\end{equation}

\medskip 

\paragraph{Mixed Anticommutators.}
Finally we analyze the anticommutators $[ {\bs a}_-(f) \, , \, {\bs a}_-^*(g) ]_+$.
At a first stage we consider $f,g \in \efh_\circ$ without further hypothesis,  
and vectors of the type $v_- =  w_-A \in \efh^n_-$,
with $w_- = w_1 \otimes_- \ldots \otimes_- w_n \in \bsh_-^n$, $A \in \mA$;
then using (\ref{eq.A0.1.PS1a''}) and (\ref{eq.A0.1.PS1a'''}) respectively, we compute
\begin{equation}
\label{eq.CAR.1a}
{\bs a}_-(f){\bs a}_-^*(g) v_- 
\ = \
\lb f,g \rb v_- - 
\sum_k (-1)^{k-1} 
\lb f,w_k \rb \, P_-^n (g \otimes w^{(k)}_-) A
\end{equation}
\begin{equation}
\label{eq.CAR.1b}
{\bs a}_-^*(g) {\bs a}_-(f) v_- 
\, = \, 
\sum_k (-1)^{k-1}  P_-^n (g \lb f,w_k \rb \otimes w^{(k)}_- ) A \, .
\end{equation}
A quick look to the previous equalities is sufficient to realize that the terms
\begin{equation}
\label{eq.CAR.1b'}
\lb f,w_k \rb \, P_-^n (g \otimes w^{(k)}_-) A
\ \ \ , \ \ \ 
P_-^n (g \lb f,w_k \rb \otimes w^{(k)}_- ) A
\end{equation}
may prevent the realization of the anti-commutation relations that one could expect.
In fact, whilst in the case $\mA = \bC$ they eliminate each other leaving only the 
term $\lb f,g \rb v_-$ present, in general they could differ because the scalar products
$\lb f,w_k \rb$ are not free to shift through the elementary tensors.
Thus we adopt the usual hypothesis $f \bowtie_\circ g$
and analyze the terms in (\ref{eq.CAR.1b'}) more in details.

\medskip 

As a first step we write as usual $f=f_i \gamma_i$, $g=g_h\gamma'_h$ and, to be concise,
we write $u_i \doteq u_{\gamma_i}$, $u'_h \doteq u_{\gamma'_h} \in \mU(\bsh)$,
$\lb f,w_k \rb = \gamma_i^* z_{ik} \in \mA(f)$ with $z_{ik} \doteq \lb f_i , w_k \rb \in \bC$.
With this notation, the hypothesis $f \bowtie_\circ g$ implies 
$[ \gamma_i , \gamma'_h ] \ = \ [ \gamma_i^* , \gamma'_h ] = 0$,
so that
\[
[ u_i \, , \, u'_h ] \ = \ [ u_i^* \, , \, u'_h ] \ = \ 0 \, .
\]
For the first expression in (\ref{eq.CAR.1b'}) we compute
\[
\begin{array}{l}
\lb f,w_k \rb \, P_-^n (g \otimes w^{(k)}_-) A \ = \\
(n-1)!^{-1} \sum_{\varrho \in \bP_{n-1,k}} \eps_\varrho \,
\gamma_i^* z_{ik} \, 
P_-^n (g_h \gamma'_h \otimes w_{\varrho(1)} \otimes \ldots \otimes w_{\varrho(n)} ) A  \ = \\
(n-1)!^{-1} \sum_{\varrho \in \bP_{n-1,k}} \eps_\varrho \,
\gamma_i^* z_{ik} \, 
P_-^n (g_h \otimes u'_h w_{\varrho(1)} \otimes \ldots \otimes u'_h w_{\varrho(n)} ) 
\gamma'_h A  \, .
\end{array}
\]
The term $P_-^n (g_h \otimes u'_h w_{\varrho(1)} \otimes \ldots \otimes u'_h w_{\varrho(n)} )$
is a linear combination of terms of the type
\[
\eps_\pi \, u'_h w_{\pi\rho(l)} \otimes \ldots g_h \ldots \otimes u'_h w_{\pi\rho(m)} \, 
\gamma'_h A \, ,
\]
where $\pi$ is any permutation of the terms in the argument of $P_-^n$ and
$g_h$ appears in position $\pi(1)$.
Applying the operators $\lb f,w_k \rb = \gamma_i^* z_{ik}$, we get terms of the type
%
%
\[
\begin{array}{l}
\eps_\pi \, 
\gamma_i^* z_{ik} u'_h w_{\pi\rho(l)} \otimes \ldots g_h \ldots \otimes u'_h w_{\pi\rho(m)} \, 
\gamma'_h A \ = \\
\eps_\pi \, 
u_i^* u'_h w_{\pi\rho(l)} \otimes \ldots \gamma_i^* g_h \gamma'_h \ldots \otimes w_{\pi\rho(m)}
z_{ik} A  \ = \\
\eps_\pi \, 
u_i^* u'_h w_{\pi\rho(l)} \otimes \ldots \gamma_i^* g \ldots \otimes w_{\pi\rho(m)}
z_{ik} A  \ = \\
\eps_\pi \, 
u_i^* u'_h w_{\pi\rho(l)} \otimes \ldots g \gamma_i^* \ldots \otimes w_{\pi\rho(m)}
z_{ik} A  \ = \\
\eps_\pi \, 
u_i^* u'_h w_{\pi\rho(l)} \otimes \ldots g_h \ldots \otimes u'_h u_i^* w_{\pi\rho(m)}
z_{ik} \gamma_h \gamma_i^* A  \ = \\
\eps_\pi \, 
u_i^* u'_h w_{\pi\rho(l)} \otimes \ldots g_h \ldots \otimes u_i^* u'_h  w_{\pi\rho(m)}
z_{ik} \gamma_i^* \gamma_h A  \, .
\end{array}
\]
In conclusion,
\[
\begin{array}{l}
\lb f,w_k \rb \, P_-^n (g \otimes w^{(k)}_-) A \ = \\
(n-1)!^{-1} \sum_{\varrho,\pi} \eps_\varrho\eps_\pi  \,
u_i^* u'_h w_{\pi\rho(l)} \otimes \ldots g_h \ldots \otimes u_i^* u'_h  w_{\pi\rho(m)}
z_{ik} \gamma_i^* \gamma_h A
\, .
\end{array}
\]
Finally we evaluate the second term in (\ref{eq.CAR.1b'}),
\[
\begin{array}{l}
P_-^n (g \lb f,w_k \rb \otimes w^{(k)}_- ) A \ = \\
(n-1)!^{-1} \sum_{\varrho \in \bP_{n-1,k}} \eps_\varrho \,
P_-^n 
(g_h \gamma'_h \gamma_i^* z_{ik} \otimes 
w_{\varrho(1)} \otimes \ldots \otimes w_{\varrho(n)} ) A \ = \\
(n-1)!^{-1} \sum_{\varrho \in \bP_{n-1,k}} \eps_\varrho \,
P_-^n 
(g_h \gamma_i^* \gamma'_h  z_{ik} \otimes 
w_{\varrho(1)} \otimes \ldots \otimes w_{\varrho(n)} ) A \ = \\
(n-1)!^{-1} \sum_{\varrho \in \bP_{n-1,k}} \eps_\varrho \,
P_-^n 
(g_h  \otimes 
u_i^* u'_h w_{\varrho(1)} \otimes \ldots \otimes u_i^* u'_h w_{\varrho(n)} ) 
z_{ik} \gamma_i^* \gamma'_h A \, = \\ 
(n-1)!^{-1} \sum_{\varrho,\pi} \eps_\varrho \eps_\pi \,
u_i^* u'_h w_{\varrho(1)} \otimes \ldots g_h \ldots \otimes u_i^* u'_h w_{\varrho(n)}
z_{ik} \gamma_i^* \gamma'_h A 
\, ,
\end{array}
\]
concluding that 
\[
\lb f,w_k \rb \, P_-^n (g \otimes w^{(k)}_-) A
\ = \ 
P_-^n (g \lb f,w_k \rb \otimes w^{(k)}_- ) A
\]
actually conspire to obtain, starting from (\ref{eq.CAR.1a}) and (\ref{eq.CAR.1b}),
the anti-commutation relations

\medskip

\begin{equation}
\label{eq.CAR.1}
[ {\bs a}_-(f) \, , \, {\bs a}_-^*(g) ]_+  \ = \ \lb f,g \rb v_-
\ \ \ , \ \ \ 
f \bowtie_\circ g \, \in \efh_\circ \, .
\end{equation}

\medskip 

\noindent

\paragraph{Norm and generalized CARs.}
As a particular case we now consider $f,g \in \bsh$
that clearly implies $f \bowtie_\circ g$;
moreover, we have $\lb f , w \rb \in \bC$ for all $w \in \bsh$,
thus the undesirable terms in (\ref{eq.CAR.1}) vanish. 
As a consequence we find
\begin{equation}
\label{eq.acr.hilb}
[ {\bs a}_-(f) \, , \, {\bs a}_-^*(g) ]_+ \, = \, \lb f,g \rb \in \bC
\ \ , \ \ 
[ {\bs a}_-^*(f) \, , \, {\bs a}_-^*(g) ]_+ \, = \, 0
\ \ , \ \ 
f,g \in \bsh \, .
\end{equation}
Now, ${\bs a}_-(f)$ and ${\bs a}_-^*(g)$ act like the usual annihilation
and creation operators when restricted to the Fock space 
$\mF_-(\bsh) \subset \efF_-(\efh)$:
we denote the corresponding restrictions by $a_-(f)$, $a_-^*(g)$
(without bold font), and note that $\| a_-(f) \| = \| a_-^*(f) \| = \| f \|$
\cite[Vol.2, Prop.5.2.2]{BR}. 
Given $w_- \in \mF_-(\bsh)$, $A \in \mA$, by right $\mA$-linearity we have
\[
{\bs a}_-(f) (w_-A) \ = \ (a_-(f)w_-)A
\ \ \ , \ \ \ 
{\bs a}_-^*(f) (w_-A) \ = \ (a_-^*(f)w_-)A
\ \ \ , \ \ \ 
f \in \bsh
\, ,
\]
so that 
\[
\begin{array}{lcl}
\| \lb {\bs a}_-(f) (w_-A) \, , \, {\bs a}_-(f) (w_-A) \rb \|  & = &
\| A^* \lb w_- , a_-^*(f) a_-(f)w_- \rb A \|  \leq \\ & \leq &
\| f \|^2 \| A^* \lb w_- , w_- \rb A \|  = \\ & = &
\| f \|^2 \| w_-A \|^2 \, ,
\end{array}
\]
having used the fact that $\| f \|^2 - a_-^*(f) a_-(f)$ is a positive operator
on $\mF_-(\bsh)$.
%
%
%
%
Thus $\| {\bs a}_-(f) \| = \| {\bs a}_-^*(f) \|  = \| f \|$
for all $f \in \bsh$.
We use this property to prove the following, more general, result:
\begin{lem}
\label{lem.ca.b}
Let $f \in \efh_\circ$. Then ${\bs a}_-(f)$ and ${\bs a}_-^*(f)$ are bounded.
\end{lem}

\begin{proof}
We start proving our assertion for the annihilation operator.
Given $f = f_i \gamma_i$, $f_i \in \bsh$, $\gamma_i \in \mG$ 
(finite sum), for the usual elementary tensors $w_-A \in \efh_-^n$ we compute
\[ 
\begin{array}{lcl}
{\bs a}_-(f) w_-A & = &
{\bs a}_-(f_i \gamma_i) w_-A \, = \\ & = &
\frac{1}{\sqrt{n}} \sum_k (-1)^{k-1} \gamma_i^* \lb f_i , w_k \rb \, w^{(k)}_-A \, = \\ & = &
\sum_i \gamma_i^* \, {\bs a}_-(f_i) w_-A \, .
\end{array}
\]
The previous relations say that ${\bs a}_-(f)$ is sum of the operators 
$\gamma_i^* {\bs a}_-(f_i)$, where $\gamma_i^*$ are regarded as unitary operators on $\efh$ and
${\bs a}_-(f_i)$ are, by the previous remarks, bounded. 
Thus we conclude that ${\bs a}_-(f)$ is bounded for $f$ (finite) linear combination 
in $\efh_\circ$.
\end{proof}

\medskip


In the following result we give a synthesis of
(\ref{eq.20a}),
(\ref{eq.20b}),
(\ref{eq.CAR.1}), and
Lemma \ref{lem.ca.b}.
%

\medskip 

\begin{thm}
\label{thm.A0.1}
Let $\efh = \bsh\mA$ be a free Hilbert $\mA$-bimodule with twist 
$u : \mG \to \mU(\bsh)$, $\mG \subseteq \mU\mA$. 
Then for any $f,g \in \efh_\circ$ the creation and annihilation operators
are bounded right $\mA$-module operators on $\efF_-(\efh)$, and
the following properties hold:
\textbf{(1)} If $f \bowtie_\circ g$, then
\begin{equation}
\label{eq.CAR.20'a}
[ {\bs a}_-(f) \, , \, {\bs a}_-(g) ]_+
\ = \ 
[ {\bs a}_-^*(f) \, , \, {\bs a}_-^*(g) ]_+ 
\  = \  
0 \, ,
\end{equation}
and
\begin{equation}
\label{eq.CAR.20'}
[ {\bs a}_-(f) \, , \, {\bs a}_-^*(g) ]_+ \, = \, \lb f,g \rb 
\in  \mA \, .
\end{equation}

\medskip

\noindent 
\textbf{(2)} If $f,g \in \bsh$, then the previous anti-commutation relations 
hold with $\lb f , g \rb \in \bC$.
\end{thm}


\medskip 

\paragraph{Dirac fields.}
Let $\mU_*(\bsh)$ denote the set of anti-unitary operators on $\bsh$,
and $\kappa = \kappa^* \in \mU_*(\bsh)$ a conjugation such that
\begin{equation}
\label{eq.k}
[ \kappa \, , \, u_\gamma ] \, = \, 0 
\ \ \ , \ \ \ 
\forall \gamma \in \mG \subseteq \mU\mA 
\, .
\end{equation}
Setting $\kappa(vA) \doteq (\kappa v)A^*$, $v \in \bsh$, $A\in \mA$,
we extend $\kappa$ to the vector space spanned by elementary tensors in $\efh$,
obtaining a densely defined antilinear map.
Note that in particular $\kappa(v\gamma) = (\kappa v)\gamma^*$,
so $\kappa$ is defined on $\efh_\circ$.
%
%
In the following result we check the compatibility of $\kappa$
with the mutual freeness relation (\ref{eq.split}):
\begin{lem}
\label{lem.conj}
The following properties hold:
\textbf{(1)} $\mA(f) = \mA(\kappa f)$ for all $f \in \efh_\circ$;
\textbf{(2)} If $f,g \in \efh_\circ$ and $f \bowtie g$, then $\lb f , 
             \kappa g \rb = \lb g , \kappa f \rb$;
\textbf{(3)} Let $f \bowtie_\circ g$;
             then $f \bowtie_\circ \kappa g$, $\kappa g \bowtie_\circ f$ and $\kappa f \bowtie_\circ \kappa g$.
\end{lem}

\begin{proof}
\textbf{(1)} We can write $f = w_i A_i$ with $\{ w_i \}$ an orthogonal base in $\bsh$
and $A_i \doteq \lb w_i , f \rb \in \mA$: this implies that $\mA(f)$ is generated by the set $\{ A_i \}$.
On the other hand, $\kappa f = \kappa(w_i) A_i^*$ where also $\{ \kappa w_i \}$ is a base of $\bsh$,
implying that $\mA(\kappa f)$ is generated by $\{ A_i^* \}$.
Thus $\mA(f) = \mA(\kappa f)$ as claimed.
\textbf{(2)} Writing $g = v_h B_h$, $v_h \in \bsh$, $B_h \in \mA(g)$, we get
\[
\lb f , \kappa g \rb \ = \ 
A_i^* B_h^* \lb w_i , \kappa v_h \rb \ = \ 
B_h^* A_i^* \lb v_h , \kappa w_i \rb \ = \ 
\lb v_h B_h , (\kappa w_i) A_i^* \rb \ = \ 
\lb g , \kappa f \rb \, ,
\]
having used the fact that $f \bowtie g$ implies $[ A_i , B_h ] = 0$.
\textbf{(3)} We write $f = f_i\gamma_i$, $g = g_h\gamma'_h$ 
and check that $f \bowtie_\circ \kappa g$:
\[
\begin{array}{lcl}
(\kappa g) \gamma_i & = &
\kappa (g_h \gamma'_h) \gamma_i \ = \ 
(\kappa g_h) \gamma_i {\gamma'_h}^*  \ = \   
\gamma_i (u_{\gamma_i}^* \kappa g_h ) {\gamma'_h}^*  \\ & = &
\gamma_i ( \kappa u_{\gamma_i}^*  g_h ) {\gamma'_h}^* \ = \ 
\gamma_i ( \kappa g_h ) {\gamma'_h}^*  \ = \ \gamma_i (\kappa g) \, .
\end{array}
\]
The other cases are verified in an analogous way, so the Lemma is proved.
\end{proof}

\medskip 

A \emph{Dirac triple over $\mA$}, written $( \efh , u , \kappa )$, is given by
a free Hilbert $\mA$-bimodule $\efh = \bsh\mA$ with $\mG$-twist $u$,
and a conjugation $\kappa \in \mU_*(\bsh)$ fulfilling (\ref{eq.k}).
%
The (self-dual) \emph{Dirac field} associated with
$( \efh , u , \kappa )$
is defined by
\begin{equation}
\label{eq.FF01}
\hat{\psi}(f) \ \doteq \, \frac{1}{\sqrt{2}} \, ( {\bs a}_-^*(f) + {\bs a}_-(\kappa f) ) 
\ \ , \ \ 
f \in \efh_\circ \subseteq \efh \, .
\end{equation}
It yields operators $\hat{\psi}(f) \in \efB(\efF_-(\efh))$, and
by Lemma \ref{lem.Gtwist3} we have
\begin{equation}
\label{eq.FF01.sd}
\hat{\psi}^*(f) \ = \ \hat{\psi}(\kappa f) \, .
\end{equation}
By applying
Lemma \ref{lem.Gtwist4}, Theorem \ref{thm.A0.1} and Lemma \ref{lem.conj}
we obtain

\begin{equation}
\label{eq.FF02.sd}
\left\{
\begin{array}{lll}
[ \hat{\psi}(f) \, , \, \hat{\psi}(g) ]_+ 
\ = \ 
\lb \kappa f,g \rb \in \mA 
& , & 
f \bowtie_\circ g \in \efh_\circ \, ,
\\ \\
\gamma \hat{\psi}(w) \ = \ \hat{\psi}(u_\gamma w) \gamma
& , &
w \in \bsh \, , \, \gamma \in \mG \subseteq \mU\mA \, .
\end{array}
\right.
\end{equation}

\noindent We denote the C*-algebra generated by the operators 
$\hat{\psi}(f) , \gamma \in \efB(\efF_-(\efh))$, $f \in \efh_\circ$, $\gamma \in \mG$,
by $\mF_{\efh,u,\kappa}$, and call it the \emph{field C*-algebra} of $(\efh,u,\kappa)$.
By construction $\mF_{\efh,u,\kappa}$ contains $\lambda(\mA)$ and the CAR algebra $\mC_\bsh$,
and fulfils the relations (\ref{eq.FF01.sd}) and (\ref{eq.FF02.sd})
{\footnote{
Relations similar to (\ref{eq.FF02.sd}) appeared in \cite{Her96,Her98} 
in the special case where $\mA$ is a Weyl algebra describing 
suitable asymptotic configurations of the electromagnetic field. 
Whilst we use the C*-norm induced by our Fock bimodule, 
the C*-algebra of the above reference is endowed with a maximal C*-norm,
that is non-trivial because a (non-separable) representation is exhibited.
}}.
Any state $\omega \in \mS(\mA)$ induces a Hilbert space representation 
of $\mF_{\efh,u,\kappa}$, defined as in (\ref{eq.0.07'}).
We remark that at the abstract level at which we worked 
no topology has been defined on $\efh_\circ$, thus no continuity property
is required for $\hat{\psi}(f)$ at varying of $f \in \efh_\circ$.

\paragraph{A class of fixed-time models.}
We briefly present a family of models for the notion of Dirac triple
and the associated Dirac field. A more detailed exposition of these and other models is
postponed to a future publication.

\medskip 

We start by considering the symplectic space $\mS$
given by pairs of compactly supported test functions 
$s = (s_0 , s_1) \in \mS(\bR^3) \oplus \mS(\bR^3)$,
with symplectic form
\[
\eta(s,s') \ \doteq \ \int ( s_1  s'_0 - s_0 s'_1 )
\]
(Lebesgue measure) and the associated Weyl C*-algebra $\mW$.
It is readily seen that $\mW$ is the C*-algebra associated 
to the restriction at a fixed time of the free scalar field,
with 
\[
W(s) \ = \ e^{i (\phi(s_0) + \dot\phi(s_1) )} \, .
\]
Here, the field $\phi(s_0)$ and its conjugate $\dot\phi(s_1)$, $s_0 , s_1 \in \mS(\bR^3)$, 
are the initial conditions at time $t_0$ of the free scalar field \cite[\S 8.4.A]{BLOT}.
As explained in Example \ref{ex.twist2}, any unitary morphism of $\mS$ as an additive group
yields a unitary morphism of the group $\mG$ generated by Weyl unitaries
and phases.

\medskip 

We then consider the Hilbert spaces 
$\bsh_+ \doteq L^2(\bR^3,\bC^4)$, $\bsh_- \doteq L^2(\bR^3,\bC^{4,*})$,
where $\bC^{4,*}$ is the conjugate space, and define $\bsh \doteq \bsh_+ \oplus \bsh_-$
with conjugation
\[
\kappa \in \mU_*(\bsh)
\ \ \ , \ \ \ 
\kappa(w_+ \oplus \bar{w}_- ) \, \doteq \,  w_- \oplus \bar{w}_+
\]
(here $\bar{w}$ is the conjugate map $\bar{w}(w') \doteq \lb w,w' \rb$
defined by $w \in \bsh_+$).
Given a tempered distribution $\sigma \in \mS'(\bR^3)$, we consider the unitary morphism
\[
u_\sigma : \mS \to \mU(\bsh) 
\ \ , \ \
u_{\sigma,s}(w_+ \oplus \bar{w}_-) \, \doteq \, 
e^{-i \sigma \star s_0}w_+ \oplus e^{i \sigma \star s_0}\bar{w}_- 
\ \ , \ \ 
w \in \bsh \, ,
\]
where $\sigma \star s_0 \in C^\infty(\bR^3)$ is the convolution (note that $s_1$ does not come into play).
It is then clear that $\kappa$ fulfils (\ref{eq.k}).
%

\medskip

We are now in condition to form the free Hilbert bimodule $\efh = \bsh\mW$ carrying the twisting
defined by $u_\sigma$; we have $\efh = \efh_+ \oplus \efh_-$ 
with obvious meaning of the symbols, and any $(\efh,u_\sigma,\kappa)$ is a Dirac triple over $\mW$.
With this input, we have the self dual Dirac field
$\hat{\psi}(h)$, $h \in \efh_\circ$,
from which for convenience we extract the electron field
$\psi(f) \doteq \hat{\psi}(f \oplus 0)$,
$f \in \efh_{+,\circ}$,
fulfilling the relations

\medskip 

\begin{equation}
\label{eq.FF02.sd.ex}
\left\{
\begin{array}{lll}
\left[ \psi^*(f) \, , \, \psi(g) \right]_+ 
\ = \ 
\lb f,g \rb \in \mW
& , & 
f \bowtie_\circ g \in \efh_{+,\circ} \, ,
\\ \\
W(s) \psi(w) \ = \ \psi(e^{-i \sigma \star s_0} w) W(s)
& , &
w \in \bsh_+ \, , \, s \in \mS \, .
\end{array}
\right.
\end{equation}

\medskip 

\noindent 
We denote the associated field C*-algebra by $\mF_\sigma$.
It is endowed with the gauge action
\[
\beta : \bU(1) \to \Aut \mF_\sigma 
\ \ \ , \ \ \ 
\beta_z(\psi(f)) \, \doteq \, \bar{z} \psi(f)
\, .
\]
Now, given 
$f = f_i W(s_i) \in \efh_{+,\circ}$,
$f_i \in \bsh_+$, $s_i \in \mS$,
we define the support $\supp(f) \subset \bR^3$ as the union of the "fermionic" and "bosonic" supports
\[
\supp_\psi(f) \doteq \cup_i \supp(w_i)
\ \ \ , \ \ \ 
\supp_\mW(f) \doteq \cup_i \supp(s_i)
\, ,
\]
and introduce the C*-algebras 
$\mF_\sigma(A)$
generated by those $\psi(f)$ having support in the open set $A \subset \bR^3$.
There are two subnets $\mC_\bsh$ and $\mW$ of $\mF_\sigma$, the first defined 
by the operators $\psi(w)$, $w \in \bsh_+$, and the second given by 
the unitaries $W(s)$, $s \in \mS$: the two subnets are defined by the free Dirac field 
and the free scalar field respectively. In particular,
\begin{equation}
\label{eq.mod00}
[ \psi(w) , \psi(w')   ]_+ \ = \ 
[ \psi^*(w_1) , \psi(w_2) ]_+ \ = \ 0
\, ,
\end{equation}
for all $w,w' \in \bsh_+$ and $w_1 , w_2 \in \bsh_+$ 
such that $\supp(w_1) \cap \supp(w_2) = \emptyset$.
We discuss the field net $\mF_\sigma$ for several choices of 
$\sigma \in \mS'(\bR^3)$.
\begin{itemize}
    \item[1.] \emph{$\sigma$ is the Dirac delta at the origin}.
          In this case $\sigma \star s_0 = s_0$ and the unitaries $W(s)$ induce 
          by adjoint action the local gauge transformations
          \[
          \psi(w)  \, \to \, \psi( e^{-i s_0}w ) \, .
          \]
          If $\supp(w) \cap \supp(s_0) = \emptyset$
          then $u_{\sigma,s} w = w$ and $[ \psi(w) , W(s) ] = 0$. 
          Therefore the subnets $\mC_\bsh$ and $\mW$ are relatively local, that is, 
          $\mC_\bsh(A) \subset \mW(B)'$ for $A \cap B = \emptyset$,
          and $\mF_\sigma$ is local in the sense that it fulfils normal commutation relations. 
          For $A \cap B \neq \emptyset$ the second of (\ref{eq.FF02.sd.ex}) in general holds
          with $u_{\sigma,s}w \neq w$ and $[ W(s) , \psi(w) ] \neq 0$.
          A sufficient condition to having $f \bowtie_\circ g$ is 
          $\supp_a(f) \cap \supp_b(g) = \emptyset$
          for all combinations in $a,b = \psi , \mW$ different from $a=b=\psi$:
          in this case
          \begin{equation}
          \label{eq.mod02}
          [ \psi^*(f) , \psi(g) ]_+ \, = \, 
          \sum_{ij} \lb f_i , g_j \rb W(s'_j - s_i) \, \in \mW \, ,
          \end{equation}
          having written $g = g_j W(s'_j)$.
          If $f,g$ are not mutually free, then terms of the type
          (\ref{eq.CAR.1a}-\ref{eq.CAR.1b}), that are not in $\mW$,
          appear in the corresponding anticommutator.
    \item[2.] \emph{$\sigma$ has support with a non-empty interior and contained
          in the 3--ball $B_r$, $r \in ( 0,\infty]$}.
          In this case $\supp(\sigma \star s_0) \subset \supp(s_0) + B_r$.
          We may have $u_{\sigma,s}w = e^{-i\sigma \star s_0}w \neq w$ 
          even for $\supp(s) \cap \supp(w) = \emptyset$, and
          \[
          W(s) \psi(w) \ = \ \psi(e^{-i\sigma \star s_0} w) W(s)
          \]
          implying that in general $\mC_\bsh$ and $\mW$ are not relatively local.
          If $f,g \in \efh_{+,\circ}$, then a sufficient condition to having 
          $f \bowtie_\circ g$
          is that $(\supp_a(f) + B_r) \cap (\supp_b(g) + B_r) = \emptyset$ for 
          $(a,b) \neq (\psi,\psi)$; 
          in that case (\ref{eq.mod02}) holds. Again, terms of the type
          (\ref{eq.CAR.1a}-\ref{eq.CAR.1b}) appear for $f,g$ not mutually free.
          %
          %
          %
    \item[2.1]  \emph{$\sigma$ is the fundamental solution of 
         the Poisson equation} \cite[\S 9.4]{Con}. In this case
         \[
         (\sigma \star s_0)(\bsx) \, = \, 
         \frac{1}{4\pi} \int \frac{1}{|\bsx -\bsy|} s_0(\bsy) \, d^3\bsy 
         \]
         in general has non-compact support and is non-constant
         (for example, take $s_0$ a non-negative bump function supported around the origin). 
         $\mW$ is not relatively local both to $\mC_\bsh$ and 
         the fixed-point subnet $\mC_\bsh^\beta$;
         in particular,
         \begin{equation}
         \label{eq.mod03}
         W(s) \psi(w_1) \psi^*(w_2) \ = \ 
         \psi(e^{-i\sigma \star s_0} w_1) \, \psi^*(e^{i\sigma \star s_0} w_2) W(s)
         \end{equation}
         even when $\supp(s)$ is disjoint from the supports of $w_1 , w_2$.
         %
         %
    \item[2.2.] \emph{$\sigma$ is the Lebesgue measure}.
          In this case
          \[
          (\sigma \star s_0)(\bsx) \, = \, 
          \int s_0(\bsx - \bsy) \, d^3\bsy \, = \, 
          - \lb s_0 \rb \, \doteq \, 
          - \int s_0
          \]
          is constant and 
          \begin{equation}
          \label{eq.mod01}
          W(s) \psi(w) \ = \ e^{i\lb s_0 \rb} \psi(w) W(s)
          \end{equation}
          for all $s \in \mS$ and $w \in \bsh_+$. 
          The above equality has a two-fold interpretation. 
          The first is that the Weyl unitaries $W(s)$ induce the global
          gauge transformations $e^{i \lb s_0 \rb}$ on the charged fields $\psi(w)$.
          The second is that the operators $\psi(w)$
          intertwine the identity and the automorphism
          $\alpha \in \Aut \mW$, $\alpha(W(s)) \doteq e^{-i \lb s_0 \rb}W(s)$,
          %
          %
          so that (\ref{eq.mod01}) becomes
          \begin{equation}
          \label{eq.mod01'}
           \alpha(W(s)) \psi(w) \ = \ \psi(w) W(s)  \, .
          \end{equation}
          The net $\mF_\sigma$ is not local, in fact
          $\mC_\bsh(A)$ and $\mW(B)$ are never one in the commutant of the other.
          Yet by (\ref{eq.mod01}) we have that 
          $\mW$ is in the commutant of $\mC_\bsh^\beta$
          and
          $\mC_\bsh$ is in the commutant of the subalgebra $\mW^0$ generated by test functions $s$ 
          with $\lb s_0 \rb = 0$.
          A local, gauge-invariant subnet $\mA$ of $\mF_\sigma$ is the one generated by operators of the type
          $A_{s,w_1,w_2}  \doteq  \psi(w_1) W(s) \psi^*(w_2)$,
          $w_1 , w_2 \in \bsh_+  , s \in \mS$,
          in fact 
          \[
          [ \, A_{s,w_1,w_2} \, \, , \, A_{s',w'_1,w'_2} \, ] \, = \, 0
          \]
          for
          $
          ( \supp(s) \cup \supp(w_1) \cup \supp(w_2) ) 
          \cap
          ( \supp(s') \cup \supp(w'_1) \cup \supp(w'_2) ) 
          =
          \emptyset
          $,
          having used (\ref{eq.mod00}), (\ref{eq.mod01}) and
          $[ W(s) , W(s') ] = 0$.
          Note that $\mA$ is local both to $\mC_\bsh^\beta$ and $\mW$.
 \end{itemize}

\section{Conclusions}
\label{sec.C}

In the present paper we presented a construction based on Hilbert bimodules,
in which the spatial tensor product of a CAR algebra by a C*-algebra 
is replaced by a twisted product.
This allows to construct field systems with non-trivial commutation relations, as in 
(\ref{eq.FF02.sd.ex}) and (\ref{eq.mod01'}).
The technical obstacles concerning tensor products and (the absence of) 
permutation symmetry in Hilbert bimodules have been overcome by introducing 
the notion of twist, which yields a class of left actions for which
these drawbacks are under control
{\footnote{
We remark that the same technique may be used
to construct bosonic Fock bimodules and the corresponding fields: 
in such a scenario, the construction in \cite{Ske98} would be analogous to a 
bosonic Fock $\mA$-bimodule, with $\mA$ the finite $d$-dimensional Weyl algebra, 
$\bsh = L^2(\bR^d)$ and $\efh = \bsh\mA$ endowed with the trivial left 
action in the sense of the present paper.}}.


\medskip 

The models presented in the previous section are elementary,
yet they pose questions that in our opinion deserve to be discussed.
For example, the model 1 exhibits Weyl unitaries that induce local gauge transformations,
thus in regular representations of $\mW$ we expect to find bosonic fields 
assuming the role usually played by the zero components of the Dirac current 
\cite[\S 4.6.1]{Strocchi} or the "longitudinal photon field" in Gupta-Bleuer gauge 
\cite[\S 7.3.2]{Strocchi}.
In the model 2.1, $\sigma$ is related to electrostatic potentials
and not surprisingly it poses the problem of extracting a \emph{local} observable subnet
from which the initial (non-local) field net should be reconstructed.
%
%
In this regard, the model 2.2 provides a simple illustration of the fact that this 
problem can be successfully solved in specific situations.
%

\medskip 

As a final remark,
we point out that the physical understanding of the notion of twist
is a topic that has not been discussed in the present paper,
in which we used this object as a mathematical input.
The correct interpretation should be obtained by a deeper discussion of our models,
especially in regular representations of $\mW$ \cite{Vas}.
In this regard, working in a fixed-time \emph{r\'egime} allows
to avoid complications and easily produce examples, 
yet our aim is to construct and discuss models in Minkowski space, 
entering in this way in an explicitly relativistic scenario \cite{Vas2}.


\end{document}